\documentclass[reqno,a4paper]{amsart}
\usepackage{amsmath}
\usepackage{amssymb}
\usepackage{amsthm}
\usepackage{mathrsfs}
\newtheorem{thm}{Theorem}[section]
\newtheorem{lem}[thm]{Lemma}
\newtheorem{cor}[thm]{Corollary}
\usepackage{color}
\usepackage{mathtools}
\mathtoolsset{showonlyrefs=true}
\numberwithin{equation}{section}
\newcommand{\dis}{\displaystyle}
\usepackage{hyperref}
\usepackage{marginnote}
\usepackage{color}

\begin{document}

\title[Boltzmann equation in velocity-weighted Chemin-Lerner spaces]{Solution to the Boltzmann equation in velocity-weighted Chemin-Lerner type spaces}

\author{Renjun Duan}
\address{Department of Mathematics, The Chinese University of Hong Kong, Shatin, Hong Kong}
\email{rjduan@math.cuhk.edu.hk}

\author{Shota Sakamoto}
\address{Graduate School of Human and Environmental Studies,
Kyoto University,
Kyoto, 606-8501, Japan}
\email{sakamoto.shota.76r@st.kyoto-u.ac.jp}



\date{\today}

\begin{abstract}
In this paper we study the Boltzmann equation near global Maxwellians in the $d$-dimensional whole space. A unique global-in-time mild solution to the Cauchy problem of the equation is established in a Chemin-Lerner type space with respect to the phase variable $(x,v)$. Both hard and soft potentials with angular cutoff are considered. The new function space for global well-posedness is introduced to essentially treat the case of soft potentials, and the key point is that the velocity variable is taken in the weighted  supremum norm, and the space variable is in the $s$-order Besov space with $s\geq d/2$ including the spatially critical regularity. The proof is based on the time-decay properties of solutions to the linearized equation together with the bootstrap argument. Particularly, the linear analysis in case of hard potentials is due to the semigroup theory, where the extra time-decay plays a role in coping with initial data in $L^2$ with respect to the space variable.  In case of soft potentials, for the time-decay of linear equations we borrow the results based on the pure energy method and further extend them to those in $L^\infty$ framework through the technique of $L^2$--$L^\infty$ interplay. In contrast to hard potentials, $L^1$ integrability in $x$ of initial data is necessary for soft potentials in order to obtain global solutions to the nonlinear Cauchy problem. 
\end{abstract}

\maketitle

\setcounter{tocdepth}{1}
\tableofcontents
\thispagestyle{empty}

\section{Introduction}

\subsection{Setting of problem} 
We consider the following Cauchy problem on the Boltzmann equation 
\begin{equation}\label{eqn: BE}
\left\{\begin{aligned}
\dis\partial_t F (t,x,v)+v\cdot\nabla_x F (t,x,v) &=Q(F,F)(t,x,v),\\
\dis F(0,x,v)&=F_0(x,v),
\end{aligned}\right.
\end{equation}
with $(t,x,v)\in \mathbb{R}_+\times \mathbb{R}^d\times\mathbb{R}^d$,
where the bilinear collisional operator $Q$ is defined by
\begin{align*}
Q(F,G)(v)=\int_{\mathbb{R}^d}\int_{\mathbb{S}^{d-1}} B(v-v_*,\omega) (F_*'G'-F_*G)\,d\omega dv_*.
\end{align*}
Here we have used the conventional abbreviation $F'_*=F(v'_*)$, $G'=G(v')$, $F_*=F(v_*)$, and $G=G(v)$, where the pre-collisional velocities $(v, v_*)$ and the post-collisional ones $(v', v'_*)$ satisfy
\begin{align*}
v'=v-(v-v_*)\cdot \omega\omega,\quad 
v'_*=v_*+(v-v_*)\cdot \omega\omega
\end{align*}
for $\omega \in \mathbb{S}^{d-1}$, according to
the conservation laws of momentum and energy
\begin{align}\label{eqn: Conservation Laws}
v+v_*=v'+v'_*,\quad 
\vert v\vert^2+\vert v_*\vert^2=\vert v'\vert^2+\vert v'_*\vert^2.
\end{align}
We assume that the collision kernel $B$ 
takes the form of
\begin{align*}
B(v-v_*, \omega)=  \vert v-v_*\vert^\gamma b_0(\cos\theta),
\end{align*}
where $-d<\gamma\le 1$, and $\theta\in [0,\pi]$ is defined through the formula
$\cos\theta=\omega \cdot (v-v_*)/ \vert v-v_*\vert$.
We write $b_0(\theta)$ as a shorthand for $b_0(\cos\theta)$, and assume that $b_0(\theta)$ is nonnegative and satisfies
$$
0\leq b_0(\theta)\leq C |\cos\theta|,
$$ 
with a generic constant $C>0$. Then the classical Grad's angular cutoff assumption is satisfied under the above condition.
For brevity we call hard potentials for $0\le \gamma\le 1$, and soft potentials for $-d<\gamma<0$.

The goal of this paper is to look for solutions to the Cauchy problem \eqref{eqn: BE} near the equilibrium. Note that the normalized global Maxwellian
\begin{align*}
M=M(v):=\frac{1}{(2\pi)^{d/2}} 
e^{-\vert v\vert^2/2}
\end{align*}
is a steady solution to \eqref{eqn: BE} due to \eqref{eqn: Conservation Laws}. Therefore, we consider the perturbation $f=f(t,x,v)$ by 
$F:=M+M^{1/2}f$,
and reformulate the Cauchy problem \eqref{eqn: BE} as  
\begin{align}\label{eqn: BE near M}
\begin{cases}
\partial_t f +v\cdot \nabla_x f +Lf=\Gamma(f,f),\\
f(0,x,v)=f_0(x,v):=M^{-1/2}[F_0(x,v)-M].
\end{cases}
\end{align}
Here $L$ and $\Gamma$ are the linear and nonlinear parts of $Q$, respectively, defined by
\begin{align*}
Lf=-M^{-1/2} [Q(M, M^{1/2} f)+Q(M^{1/2}f, M)],\quad \Gamma(f,g)=M^{-1/2}Q(M^{1/2}f, M^{1/2} g).
\end{align*}
Under the angular cutoff assumption, it is well known that $L$ can be decomposed as 
$L=\nu - K$.
Here $\nu=\nu(v)$ is a velocity multiplication operator 
satisfying 
\begin{align*}
 \nu_0 (1+\vert v\vert^2)^{\gamma/2}\le \nu(v)\le \nu_1 (1+\vert v\vert^2)^{\gamma/2}
\end{align*}
for all $v\in \mathbb{R}^d$, where $0<\nu_0 \le \nu_1$ are constants independent of $v$. Especially, $\nu$ has a strictly positive lower bound for $0\le \gamma\le 1$, but it is not the case when $\gamma$ is negative. We remark that this fact is often the reason why one may need to consider two cases separately. Also, $K$ is an integral operator in the form of
\begin{align*}
Kf(v)=\int_{\mathbb{R}^d} k(v,\eta) f(\eta)\,d\eta
\end{align*}
for a real-valued symmetric function $k(\cdot,\cdot)$.

\subsection{Main results}
We shall state the main theorems of this paper. In order to do so, we first clarify in what sense $f(t,x,v)$ is 
a solution to the Cauchy problem \eqref{eqn: BE near M}. In fact, the mild solution $f(t,x,v)$ to \eqref{eqn: BE near M} is defined as the following integral form:
\begin{align*}
f(t,x,v) = &e^{-\nu(v)t}f_0(x-vt,v)+\int^t_0 e^{-\nu(v)(t-s)} (Kf) (s,x-(t-s)v, v)\,ds\\
&+\int^t_0 e^{-\nu(v)(t-s)} \Gamma(f,f) (s,x-(t-s)v, v)\,ds,
\end{align*}
for $t\ge 0$, $x$, $v\in \mathbb{R}^3$.
In what follows, for a Banach space $X$ and a nonnegative constant $\alpha\geq 0$ we define 
\begin{align}
\label{def.Xnorm}
|\!|\!| f |\!|\!|_{\alpha, X}=\sup_{t\ge 0} (1+t)^\alpha \Vert f(t)\Vert_X,
\end{align}
for a $X$-valued function $f(t)$ on the real half line $0\leq t<\infty$, and for any Banach spaces $X$ and $Y$, the norm $\Vert \cdot \Vert_{X\cap Y}$ means $\Vert\cdot\Vert_X+\Vert\cdot\Vert_Y$. For more notations of function spaces, especially Besov and Chemin-Lerner type spaces, readers may refer to 
the next preliminary section. 

For the hard potential case, the main result is stated as follows. 

\begin{thm}\label{thm: Solution for the hard potential case}
Assume $d\ge 1$, $0\le \gamma\le 1$, $q \in [1, \min(d,2)]$, $s\ge d/2$, and $\beta >\gamma+d/2$. Then there exist positive constants $\varepsilon>0$ and $C>0$ such that if initial data $f_0$ satisfies
\begin{align*}
\Vert f_0\Vert_{\tilde{L}^\infty_\beta (B^s_{2,1})\cap L^2_v L^q_x} \le \varepsilon,
\end{align*}
then the Cauchy problem  \eqref{eqn: BE near M}  admits a unique global mild solution $f(t,x,v) \in L^\infty(0,\infty; \tilde{L}^\infty_\beta (B^s_{2,1}))$ satisfying
\begin{align*}
|\!|\!| f |\!|\!|_{\alpha, \tilde{L}^\infty_\beta (B^s_{2,1})} \le C \Vert f_0\Vert_{\tilde{L}^\infty_\beta (B^s_{2,1})\cap L^2_v L^q_x},
\end{align*}
where $\alpha=d/2(1/q-1/2)$.
\end{thm}

For the soft potential case, we also have a similar result stated in the following

\begin{thm}\label{thm: Solution for the soft potential case}
Assume $d\ge 3$,  $-d<\gamma<0$, 
$s\ge d/2$, $\sigma=d\vert \gamma\vert/4$, 
and $\beta>\sigma_+ +d/2$, where $\sigma_+$ denotes $\sigma+\delta$ for an arbitrary small constant $\delta>0$. 
Then there exist positive constants $\varepsilon>0$ and $C>0$ such that if initial data $f_0$ satisfies
\begin{align*}
\Vert f_0\Vert_{\tilde{L}^\infty_{\beta+\sigma}(B^s_{2,1})\cap L^2_{\sigma_+} L^1_x} \le  \varepsilon ,
\end{align*}
then the Cauchy problem \eqref{eqn: BE near M} admits a unique global mild solution $f(t,x,v) \in L^\infty(0,\infty; \tilde{L}^\infty_\beta (B^s_{2,1}))$
satisfying 
\begin{align*}
|\!|\!| f|\!|\!|_{d/4, \tilde{L}^\infty_\beta (B^s_{2,1})} \le C \Vert f_0\Vert_{\tilde{L}^\infty_{\beta+\sigma}(B^s_{2,1})\cap L^2_{\sigma_+} L^1_x}.
\end{align*}
\end{thm}

\subsection{Remarks on the results}\label{sub1.3}
Here we would make a few remarks on the main theorems stated above. First, the main motivation of this paper is to treat global well-posedness in spatially critical Besov-type spaces especially in the case when the intermolecular interaction potential is very soft, that is the case of $-3<\gamma<-2$ including the situation where $\gamma$ can be close to $-3$. Indeed, in \cite{MS} regarding the angular non-cutoff Boltzmann equation, the following condition 
$$
\gamma>\max\{-3,-3/2-2\textsf{s}\},\quad 0<\textsf{s}<1,
$$
is required to establish the global well-posedness, where $\textsf{s}$ is a singularity parameter of $b_0(\theta)$ near $\theta =\pi/2$. One may carry out the same proof for the angular cutoff case which formally corresponds to the limiting situation $\gamma>-3/2$ as $\textsf{s}\to 0$. Thus the existing approaches as in \cite{DLX} and \cite{MS} cannot be directly applied to treat the case of $-3<\gamma\leq -3/2$. The main reason for this obstacle is that the solution space $\tilde{L}^2_v (B^s_{2,1})$, particularly $\tilde{L}^2_v$ with respect to velocity variable $v$,  was used in those works, so that $\gamma$ cannot be below $-3/2$ in order for the nonlinear term $\Gamma(f,f)$ to be controlled by the product of two $\tilde{L}^2_v$ norms. To overcome the difficulty, we shall use the velocity-weighted function space $\tilde{L}^\infty_\beta (B^s_{2,1})$ in the $L^\infty$ setting. Indeed, by a suitable choice of $\beta$, $\tilde{L}^\infty_\beta (B^s_{2,1})$ is  an algebra so that the nonlinear term $\Gamma(f,f)$ can be defined.

Second, as pointed out in \cite{DLX},
the most interesting value of the index $s$ under consideration is $d/2$.
One may not expect to take this value if one seeks a solution in the usual Sobolev space $H^{d/2}$, since $H^{d/2}$ is not embedded into $L^\infty$ while the Besov space $B^{d/2}_{2,1}$ is. In such sense, the regularity index $s=d/2$ is said to be spatially critical. 
However, we remark that it is still a problem to justify any blow-up of solutions in the function space either $\tilde{L}^2_v(B^s_{2,1})$ or  $\tilde{L}^\infty_\beta (B^s_{2,1})$ with $s<d/2$.

Third, the Chemin-Lerner type space $\tilde{L}^\infty_\beta(B^s_{2,1})$ is endowed with a stronger topology than the formerly used velocity-weighted Sobolev space $L^\infty_\beta (H^s)$; see \cite{U} and \cite{UA}. Indeed, 
thanks to the equivalence of $H^s$ and $B^s_{2,2}$, 
one has 
\begin{align*}
\Vert f \Vert_{L^\infty_\beta (H^s)} &=\sup_v \langle v\rangle^\beta \Vert f(v)\Vert_{H^s}
\le C \sup_v\langle v\rangle^\beta \Vert f(v) \Vert_{B^s_{2,2}}\\
&\le C\sup_v\langle v\rangle^\beta \Vert f(v) \Vert_{B^s_{2,1}}
\le C\sum_{j\ge -1} 2^{js}\sup_v \langle v\rangle^\beta \Vert \Delta_j
f(v) \Vert_{L^2_x}=C\Vert f\Vert_{\tilde{L}^\infty_\beta (B^s_{2,1})}.
\end{align*} 
Moreover, as seen from the proof of two main theorems later on, we remark without any proof that under the same conditions, the global existence of solutions can also be obtained in terms of the following stronger norm
$$
\sum_{j\geq -1} 2^{js} \sup_{0\leq t\leq T}\sup_{v\in \mathbb{R}^3} \langle v \rangle^\beta (1+t)^\alpha \|\Delta_j f(t,\cdot,v)\|_{L^2_x},
$$
with an arbitrary $T>0$. Such norm is again of the Chemin-Lerner type.

Fourth, although both the solution spaces and decay rates of the solution are the same in Theorem \ref{thm: Solution for the hard potential case} and Theorem \ref{thm: Solution for the soft potential case}, strategies of the proofs are highly contrasting. Theorem \ref{thm: Solution for the hard potential case} is shown via a time-decay property of a semigroup and an induction scheme of inequalities, from which we deduce a contraction property of a solution map. Meanwhile, the proof of Theorem \ref{thm: Solution for the soft potential case} is based on a priori estimates by the energy method and the  continuation of a local solution. The semigroup argument works only when $-1<\gamma \le 1$ for $d=3$ (see \cite[Theorem 8.2]{UA}), and this requirement is truly essential. This is the reason why we employ the energy method for the soft potential case which is the main concern of the paper.
We remark that the energy method may also work for the hard potential case if $d\ge 3$, but the details are omitted for brevity of presentation. 

Fifth, the assumption $d\ge 3$ for the soft potential case in Theorem \ref{thm: Solution for the soft potential case}  is also essential. In fact, for soft potentials, one can obtain the usual decay estimate of the semigroup for the linearized problem by the energy method, but it seems impossible to deduce an extra decay of the semigroup acting on the pure non-fluid function 
(see the second estimate in Lemma \ref{lem:time-decay of semigroup}). Therefore it is necessary in the proof of Theorem \ref{thm: Solution for the soft potential case}  to require the condition $d/4+d/4>1$ for $q=1$, 
namely $d\ge 3$, in order to make use of the enough time-decay of solutions to close the nonlinear estimates. Note that the value of $q$ could be improved to be slightly greater than $1$, but it seems hopeless for us to treat $q=2$.  
Moreover, we remark in the hard potential case that if $q=2$ then $\tilde{L}^\infty_{\beta}(B^s_{2,1})\subset L^2_vL^q_x$ holds true thanks to $\beta>d/2$, and hence the global existence stated in Theorem \ref{thm: Solution for the hard potential case} can directly follow for initial data $f_0$ small enough in the function space $\tilde{L}^\infty_{\beta}(B^s_{2,1})$ without any other restriction in contrast to the case of $1\leq q<2$. 

\subsection{Known results}
Indeed, the perturbation theory of the Boltzmann equation near global Maxwellians has been well established so far. Interested readers may refer to \cite{DLX} as well as \cite{DW} for  an almost  complete list of references on the subject. In what follows we would only mention some research works most related to our current study in this paper. First of all, semigroup theory of the Boltzmann equation with cutoff was developed first in \cite{U2} for the whole range of hard potentials $0\leq \gamma\leq 1$ and then in \cite{UA} for the partial range of soft potentials $-1<\gamma<0$. It still remains open to extend those results to the case of $-3<\gamma\leq -1$, and we remark that if it could be achieved then one can remove the extra restriction on initial data in  Theorem \ref{thm: Solution for the soft potential case}. The function space  $L^\infty (0,\infty; \tilde{L}^\infty_\beta (B^s_{2,1}))$ or $\tilde{L}^\infty_T\tilde{L}^\infty_\beta (B^s_{2,1})$ with $s\geq d/2$ introduced in this paper can be regarded as a direct generalization of the known one $L^\infty (0,\infty; L^\infty_{\beta}H^s_x)$ with $s>d/2$. To deal with soft potentials, we have used some techniques from \cite{Caf1,Caf2,DHWY, DYZ,Guo-soft,Guo-bd}. Particularly, \cite{Guo-soft} introduced the decomposition of $K$ into a {\it compact} part $K^c$ and a {\it small} part $K^m$. The $L^2$--$L^\infty$ interplay approach was first introduced in \cite{Guo-bd} for treating the Boltzmann with hard potentials, see also \cite{UY} for a different view,  and it  has been extended in \cite{DHWY} to the soft potential case. In the current work we have made use of those works to additionally take into account the time-decay property as well as estimates in the Chemin-Lerner type space.

As we have seen, there have been many known results where the Sobolev space $H^s_x$ has been utilized, on the other hand, a successful application of the Besov space to the Cauchy problem is first achieved in \cite{DLX}. Under the cutoff assumption, the authors proved global existence of a unique global solution in the space
\begin{align*}
\tilde{L}^\infty \tilde{L}^2 (B^s_{2,1}) \left((0,\infty)_t\times  \mathbb{R}^3_v\times\mathbb{R}^3_x\right), \ s\ge 3/2
\end{align*}
for the hard potential case.
Following this result,  \cite{TL} considers the problem under the same conditions in the above space replacing $B^s_{2,1}$ by $B^s_{2,r}$ with $1\le r\le 2$ and $s>3/2$.
Also, it is proved in 
\cite{MS} that 
the Cauchy problem is well-posed in the same space for the Boltzmann equation without angular cutoff.
It should be also noted that the use of the Besov space in this paper is strongly motivated by \cite{DLX}, therefore, we here provide another aspect of applications of the Besov space to the problem. Finally, we refer readers to \cite{AM}  and \cite{SS} for applications of the Besov space to the kinetic theory from different perspectives.

\subsection{Organization and notation of the paper}
The rest of this paper is organized as follows. In Section \ref{Ch: Pre}, we will define function spaces to be used throughout the paper. Some lemmas from the preceding works will be also catalogued. In Sections \ref{Ch: Hard} and \ref{Ch: Soft}, we shall show  the global existence and uniqueness of solutions with explicit  time-decay rates for both the hard and soft potential cases, respectively. In Appendix, for completeness we will prove the local-in-time existence of solutions in the soft potential case.

Throughout the paper, $C$ denotes some generic positive (generally large) constant, and may take different values in different places.  

\section{Preliminaries}\label{Ch: Pre}
In this section, we define some function spaces for later use.
We also cite some lemmas on which one may rely as a basis of the proof in the following sections. The proofs of those lemmas will be omitted for brevity; the interested readers may refer to the original paper and references therein.

For $1\le p\le \infty$, $L^p=L^p(\mathbb{R}^d)$ is the usual $L^p$-space endowed with $\Vert \cdot \Vert_{L^p}$. In this paper, integration (or supremum if $p=\infty$) is always taken over $\mathbb{R}^d$ with respect to $x$ or $v$. Thus, whenever it is obvious, we omit $\mathbb{R}^d$ in what follows. For $1\le p$, $q\le \infty$, we define
\begin{align*}
L^p_v L^q_x= L^p \left(\mathbb{R}^d_v ; L^q(\mathbb{R}^d_x)\right),\quad L^q_x L^p_v= L^q \left(\mathbb{R}^d_x ; L^p(\mathbb{R}^d_v)\right). 
\end{align*}
A velocity-weighted $L^p$ space with a weight index $\beta \in\mathbb{R}$ is defined as
\begin{align*}
L^p_\beta =\left\{f=f(v)\left| \langle \cdot\rangle^\beta f \in L^p\right.\right\},\quad \Vert f\Vert_{L^p_\beta}:= \Vert \langle \cdot \rangle^\beta f\Vert_{L^p}
\end{align*}
where $\langle v\rangle = (1+\vert v\vert^2)^{1/2}$. We remark that the weighted norm is only used for the velocity variable, and further that we often use $L^p_0$ instead of $L^p$ when  the weight index need to be emphasized.

In order to define a Besov space, we first introduce the Littlewood-Paley decomposition, cf.~\cite{BCD}.
We define $A(r,R)$ for $0<r<R$ as an annulus bounded by  a larger circle of radius $R$ and a smaller one of radius $r$ both centred at the origin, and $B_R$ for $R>0$ as a ball with radius $R$ centred at the origin. There exist radial functions $\chi$ and $\phi$ satisfying the following properties:
\begin{align*}
\chi \in C^\infty_0 (B_{4/3}),\ \phi \in C^\infty_0(A(3/4, 8/3)),\ 0\le \chi, \phi\le 1,\\
\chi (\xi)+ \sum_{j\ge 0} \phi(2^{-j}\xi) =1,\quad \xi \in \mathbb{R}^d,\\
\sum_{j\in \mathbb{Z}} \phi(2^{-j}\xi) =1,\quad \xi \in \mathbb{R}^d\backslash \{0\},\\
\vert i-j\vert \ge 2\Rightarrow \mathrm{supp}\ \phi(2^{-i}\cdot)\cap \mathrm{supp}\ \phi(2^{-j}\cdot) =\emptyset,\\
j \ge 1\Rightarrow \mathrm{supp}\ \chi \cap \mathrm{supp}\ \phi(2^{-j}\cdot) =\emptyset.
\end{align*}
The set $\tilde{A}=B_{4/3}+A(3/4, 8/3)$ is an annulus, and one has
\begin{align}\label{Nearly Orthogonal}
\vert i-j\vert \ge 5\Rightarrow 2^i \tilde{A}\cap 2^j A(3/4, 8/3) =\emptyset.
\end{align}
For a tempered distribution $f\in \mathcal{S}'(\mathbb{R}^d)$, the inhomogeneous Littlewood-Paley decomposition operators $\{\Delta_j\}_{j=-1}^\infty$ are defined as 
\begin{align*}
\Delta_j f = \phi(D) f,\ j\ge 0; \quad \Delta_{-1}f=\chi(D) f.
\end{align*}
We also define the lower-frequency cutoff operator $S_j$ as
\begin{align*}
S_jf=\sum_{-1\le i\le j-1} \Delta_i f
\end{align*}
for later use.

Now we shall define the inhomogeneous Besov space. For $s\in \mathbb{R}$ and $p,q\in [1,\infty]$, a tempered distribution $f \in \mathcal{S}'(\mathbb{R}^d)$ belongs to $B^s_{p,q}=B^s_{p,q}(\mathbb{R}^d)$ if and only if
\begin{align*}
\Vert f\Vert_{B^s_{p,q}} = \Big( \sum_{j\ge -1} 2^{jsq} \Vert \Delta_j f \Vert_{L^p}^q\Big)^{1/q}<\infty,
\end{align*}
with the usual conventions for $p$ or $q=\infty$.
It is an important and useful fact that it holds that
\begin{align*}
\Vert f\Vert_{B^s_{p,q}}\le M\Leftrightarrow \Vert \Delta_j f\Vert_{L^p}\le M 2^{-js}c_j,\ \forall j
\end{align*}
for some $c_j\in \ell^r$ with $\Vert c_j \Vert_{\ell^r}\le 1$.
We only need the pair $(p,q)=(2,1)$ in this paper, so we employ these indices in the following definition.
To simplify the notation hereafter, we write
\begin{align*}
\overline{\sum}=\sum_{j\ge -1}2^{js},
\end{align*}
which corresponds to the case $q=1$.

We shall define a Chemin-Lerner space, which can be regarded as a velocity-weighted Besov space. The following spaces play an important role throughout the paper:
\begin{align*}
\tilde{L}^2_v(B^s_{2,1}) &=\left\{f \in \mathcal{S}'(\mathbb{R}^d_x\times \mathbb{R}^d_v)\ \left|\ \Vert f\Vert_{\tilde{L}^2_v(B^s_{2,1})}=  \overline{\sum} \Vert \Delta_j f\Vert_{L^2_vL^2_x}<\infty\right.\right\},\\
\tilde{L}^\infty_\beta (B^s_{2,1}) &= \left\{f \in \mathcal{S}'(\mathbb{R}^d_x\times \mathbb{R}^d_v)\ \left|\ \Vert f\Vert_{\tilde{L}^\infty_\beta (B^s_{2,1})} = \overline{\sum} \sup_{v} \langle v\rangle^\beta \Vert \Delta_j f(\cdot, v)\Vert_{L^2_x}<\infty\right.\right\}.
\end{align*}

Next, we first collect some lemmas in the hard potential case.
For Banach spaces $X$ and $Y$, $\mathscr{B}(X,Y)$ denotes a space of linear bounded operators from $X$ to $Y$. Also, we define $\mathscr{B}(X,X)=\mathscr{B}(X)$. The following two lemmas contain some well-known facts; see \cite[Section 4]{U}, for instance.
\begin{lem}\label{lem: Ukai Prop 4.3.1}
$K\in \mathscr{B}(L^2)\cap\mathscr{B}(L^2, L^\infty_0)\cap \mathscr{B}(L^\infty_{\beta}, L^\infty_{\beta+1})$, $\beta \in\mathbb{R}$.
\end{lem}

\begin{lem}
Define the operators $A=-v\cdot\nabla_x -\nu(v)$ and $B=A+K$. Then $A$ and $B$ are generators of semigroups, with
\begin{align*}
D(B)=D(A)=\left\{ f\in L^2\ \left|\ v\cdot\nabla_x f,\ \nu f \in L^2\right.\right\}.
\end{align*}
\end{lem}

Furthermore, we cite the following lemma from \cite[Theorems 4.3.2-4.3.3]{U}.
This is a basis of the arguments for the hard potential case. 
We define $B(\xi)= -(i\xi\cdot v+ \nu(v))+K$.
\begin{lem}\label{lem: Ukai Theorems 44.3.2--4.3.3}
Assume $d\ge 1$, $0\le \gamma\le 1$. There are constants $C>0$, $\kappa_0>0$, and $0<\sigma_0<\nu_0$ such that,
\begin{enumerate}
\item for any $\xi$ with $\vert \xi \vert\le \kappa_0$, one has
\begin{align*}
e^{tB(\xi)}=\sum_{i=0}^{d+1} e^{\mu_i (\vert \xi\vert)t} P_i(\xi)+U(t,\xi),\\
\Vert U(t,\xi)\Vert_{\mathscr{B}(L^2)} \le C e^{-\sigma_0 t}, \ t\ge 0.
\end{align*}
Here, the functions $\mu_i(\cdot)$ $(i=0,1,\cdots,d+1)$ are smooth and nonpositive over $[-\kappa_0, \kappa_0]$ and satisfy $\mathrm{Re}\, \mu_i (\vert\xi\vert)\le -a\vert\xi\vert^2$ for some constant  $a>0$ independent of $i$. Also, for each $i$ it holds that 
\begin{align*}
P_i(\xi)=P^{(0)}_i(\xi/\vert\xi\vert)+\vert\xi\vert P^{(1)}_i(\xi).
\end{align*}
Here $P^{(0)}_i$ $(i=0,1,\cdots,d+1)$ are orthogonal projections, and 
\begin{align*}
P_0=\sum_{i=0}^{d+1} P^{(0)}_i(\xi/\vert\xi\vert),
\end{align*}
is the orthogonal projection from $L^2_v$ onto the null space of $L$.
\item for any $\xi$ with $\vert \xi \vert \geq \kappa_0$, one has
\begin{align*}
\Vert e^{tB(\xi)} \Vert_{\mathscr{B}(L^2)} \le Ce^{-\sigma_0 t},\  t\ge 0.
\end{align*}
\end{enumerate}
\end{lem}

Regarding the time-decay property in the soft potential case, we cite the following lemma from \cite{DYZ} (see also \cite{St}) with a slight modification of notations. Note that compared to \cite{UA}, $\gamma$ can take the full range of values for soft potentials. 

\begin{lem}\label{lem: DYZ}
Let $d\ge 3$, $-3 <\gamma < 0$, and let $\ell \ge 0$, $J>0$ be given constants. Set $\mu=\mu(v):=\langle v\rangle^{-\gamma/2}$. There is a nonnegative time-frequency functional $\mathcal{E}_\ell(t,\xi)=\mathcal{E}_\ell (\hat{f}(t,\xi))$ with
\begin{align*}
\mathcal{E}_\ell(t,\xi)\sim \Vert \mu^\ell \hat{f}(t,\xi)\Vert_{L^2_v}^2
\end{align*}
such that the solution to the Cauchy problem on the linearized homogeneous equation
\begin{align}
\begin{cases}\label{eqn: Linearized Equation}
\partial_t f +v\cdot \nabla_x f +Lf=0,\\
f(0,x,v)=f_0(x,v)
\end{cases}
\end{align}
 satisfies
\begin{align*}
\mathcal{E}_\ell(t,\xi) \le C (1+\rho(\xi)t)^{-J} \mathcal{E}_{\ell +J_+}(0,\xi),
\end{align*}
for all $t\geq 0$ and $\xi\in\mathbb{R}^d$, where $\rho(\xi)=|\xi|^2/(1+|\xi|^2)$, and $C>0$ is a generic constant.
\end{lem}

The following lemma by \cite{AMUXYwholespaceI} is also useful for the estimates of the nonlinear term.
\begin{lem}{\cite[Lemma 2.5]{AMUXYwholespaceI}}\label{equiv: v_* Integral}
Let $\rho>0$, $\delta\in \mathbb{R}$. If $\alpha >-d$ and $\beta\in\mathbb{R}$, then one has 
\begin{align*}
\int_{\mathbb{R}^d} \vert v-v_*\vert^\alpha \langle v-v_*\rangle^\beta \langle v_*\rangle^\delta e^{-\rho \vert v_*\vert^2}dv_* \sim \langle v\rangle^{\alpha +\beta}.
\end{align*}
\end{lem}

\section{Solution for the hard potential case}\label{Ch: Hard}
The aim of this section is to find a solution for the hard potential case. We start from revisiting \cite[Theorem 4.4.4]{U} so that it fits to the problem under consideration. Recall that Lemma \ref{lem: Ukai Prop 4.3.1} and Lemma \ref{lem: Ukai Theorems 44.3.2--4.3.3} are valid because they are based on $L^2_v$-analysis, not depending on the difference of a Sobolev and a Besov norm with respect to $x$. Through this section,  
we set $d\ge 1$.

\begin{lem}\label{lem:time-decay of semigroup}
For $s\in \mathbb{R}$ and $q\in [1,2]$, there is a constant $C>0$ such that it holds that
\begin{align*}
&\Vert e^{tB}f \Vert_{\tilde{L}^2_v(B^s_{2,1})} \le C (1+t)^{-\alpha} \Vert f\Vert_{\tilde{L}^2_v(B^s_{2,1}) \cap L^2_vL^q_x},\\
&\Vert e^{tB}(I-P_0)f \Vert_{\tilde{L}^2_v(B^s_{2,1})} \le C (1+t)^{-\alpha-1/2} \Vert f\Vert_{\tilde{L}^2_v(B^s_{2,1}) \cap L^2_vL^q_x},
\end{align*}
for all $t\geq 0$, where $\alpha=(d/2)(1/q-1/2)$.
\end{lem}

\begin{proof} By the Plancherel identity and Fubini's theorem,
\begin{align*}
&\Vert e^{tB}f \Vert_{\tilde{L}^2_v(B^s_{2,1})} = \overline{\sum} \Vert \Delta_j e^{tB}f\Vert_{L^2_vL^2_x}
=\overline{\sum} \Vert \phi_j e^{tB(\cdot)}\hat{f}\Vert_{L^2_\xi L^2_v}\\
&\le \overline{\sum}\left[ \left( \int_{\vert \xi \vert \le \kappa_0} \phi(2^{-j}\xi)^2 \Vert e^{tB(\xi)}\hat{f}\Vert^2_{L^2_v}d\xi \right)^{1/2}+\left( \int_{\vert \xi \vert > \kappa_0} \phi(2^{-j}\xi)^2 \Vert e^{tB(\xi)}\hat{f}\Vert^2_{L^2_v}d\xi \right)^{1/2}\right].
\end{align*}
We write the two integrals on the right as $I_{1,j}$ and $I_{2,j}$ repectively.
Lemma \ref{lem: Ukai Theorems 44.3.2--4.3.3} gives
\begin{align*}
I_{2,j} \le C \int_{\vert \xi \vert > \kappa_0} \phi(2^{-j}\xi)^2 e^{-2\sigma_0 t} \Vert \hat{f}\Vert^2_{L^2_v}d\xi \le Ce^{-2\sigma_0 t}\Vert \Delta_j f\Vert_{L^2_vL^2_x}^2,
\end{align*}
and
\begin{align*}
I_{1,j} \le C \left(\sum_{i=1}^{d+1} I^i_{1,j}+ e^{-2\sigma_0 t}\Vert \Delta_j f\Vert_{L^2_vL^2_x}^2\right),
\end{align*}
where
\begin{align*}
I^i_{1,j}= \int_{\vert \xi \vert \le \kappa_0} \phi(2^{-j}\xi)^2 e^{2\mathrm{Re}\,\mu_i(\vert \xi\vert) t} \Vert \hat{f}\Vert^2_{L^2_v}d\xi \ge 0.
\end{align*}
We remark that the infinite sum of $\big(I^i_{1,j}\big)^{1/2}$ with respect to $2^{js}$ is actually finite, up to
\begin{align*}
J=\max\{j\ge -1 \ |\ \{\vert \xi \vert \le \kappa_0\} \cap \{2^{j-1}\le \vert \xi\vert\le 2^j\} \neq \emptyset\}.
\end{align*}
Thus we shall find a uniform estimate of $I^i_{1,j}$ with respect to $i$ and $j$.

By the upper bound of $\mathrm{Re}\,\mu_i(\vert\xi\vert)$ on $\{\vert\xi\vert\le\kappa_0\}$, for the triplet $(q, q',p')$ such that $1/2p'+1/q=1$ and $1/p' +1/q'=1$ we have
\begin{align*}
I^i_{1,j}&\le \psi_0 (t)^{1/q'} \Vert \hat{f}\Vert_{L^{2p'}_\xi L^2_v}^2,
\end{align*}
where
\begin{align*}
\psi_m(t)&:=\int_{\vert \xi \vert \le \kappa_0} e^{-2a\vert \xi \vert^2 t\cdot q'} \vert \xi \vert^m d\xi=\vert \mathbb{S}^{d-1} \vert \int^{\kappa_0}_0 e^{-2q'atr^2}r^{d+m-1}dr\\
&=\frac{\vert \mathbb{S}^{d-1} \vert}{2}(2q'at)^{-(d+m)/2}\int^{2q'at\kappa_0^2}_0e^{-s}s^{(d+m)/2-1}ds
\le C(1+t)^{-(m+d)/2}.
\end{align*}
In order to estimate $\Vert \hat{f}\Vert_{L^{2p'}_\xi L^2_v}$, we apply the Minkowski integral inequality 
\begin{align*}
\Vert \hat{f}\Vert_{L^\alpha_\xi L^\beta_v}\le \Vert \hat{f}\Vert_{L^\beta_v L^\alpha_\xi}
\end{align*}
for $1\le \beta \le \alpha\le \infty$ and the inequality $\Vert \hat{f}\Vert_{L^{2p'}_\xi} \le C \Vert f\Vert_{L^q_x}$ for $q\in [1,2]$. Therefore, we obtain
\begin{align*}
I_{1,i} \le C(1+t)^{-d/2q'}\Vert f\Vert_{L^2_vL^q_x}^2=C(1+t)^{-d(1/q-1/2)}\Vert f\Vert_{L^2_vL^q_x}^2,
\end{align*}
which completes the proof of the first desired estimate.

To prove the second estimate, we first notice $(I-P_0)P_i=\vert \xi \vert P^{(1)}_i(\xi)$ in terms of  Lemma \ref{lem: Ukai Theorems 44.3.2--4.3.3}. This implies that one only has to estimate $\psi_{q'}(t)^{1/q'}$ and the similar calculations can be carried out to obtain an extra time-decay $(1+t)^{-1/2}$. This then completes the proof of Lemma \ref{lem:time-decay of semigroup}.
\end{proof}


Recall  \eqref{def.Xnorm} for the definition of the norm $|\!|\!| \cdot |\!|\!|_{\alpha, X}$. We have the following

\begin{lem}\label{lem:Time-dependent estimate of semigroup}
Let $q\in [1,2]$, $s\in \mathbb{R}$, $\beta \ge 0$, $m=0$ or $1$, and $\alpha=d/2(1/q-1/2)$. Then there is a constant $C>0$ such that it holds
\begin{align*}
|\!|\!| e^{tB}(I-P_0)^m f |\!|\!|_{\alpha+m/2, \tilde{L}^\infty_\beta (B^s_{2,1})}\le C \Vert f\Vert_{\tilde{L}^\infty_\beta (B^s_{2,1})\cap \tilde{L}^2_v(B^s_{2,1})\cap L^2_vL^q_x}.
\end{align*}
\end{lem}
\begin{proof}
First, we prove that
\begin{align*}
\Vert e^{tA}\Vert \le e^{-\nu_0 t} \quad\mathrm{in}\ \mathscr{B}(\tilde{L}^2_v(B^s_{2,1}))\ \mathrm{and}\ \mathscr{B}(\tilde{L}^\infty_\beta (B^s_{2,1})).
\end{align*}
Since $e^{tA}f = e^{-\nu(v)t} f(x-vt,v)$, it holds that 
\begin{align*}
\Vert \Delta_j e^{tA}f \Vert_{L^2_vL^2_x}^2 &= \int_{\mathbb{R}^d}\phi(2^{-j}\xi)^2 \Vert e^{-\nu(\cdot)t} e^{it(\cdot)\cdot \xi} \hat{f}(\xi,\cdot)\Vert^2_{L^2_v}d\xi \le e^{-2\nu_0 t}\Vert \Delta_j f \Vert_{L^2_vL^2_x}^2.
\end{align*}
Therefore one has 
\begin{align*}
\Vert e^{tA}f\Vert_{\tilde{L}^2_v(B^s_{2,1})} =\overline{\sum} \Vert \Delta_j e^{tA}f \Vert_{L^2_vL^2_x} \le e^{-\nu_0 t} \Vert f\Vert_{\tilde{L}^2_v(B^s_{2,1})},
\end{align*}
and the proof for the space $\tilde{L}^\infty_\beta (B^s_{2,1})$ similarly follows.

Next, thanks to the identity of operators
\begin{align}\label{eqn:operator identity}
e^{tB}= e^{tA}+\int^t_0 e^{(t-s)A}Ke^{sB}ds,
\end{align}
we are able to show that
\begin{align}\label{ineq:X-Y estimate}
|\!|\!| e^{tB} f |\!|\!|_{\alpha, X}\le C \left(\Vert f\Vert_X +|\!|\!| e^{tB} f |\!|\!|_{\alpha, Y}\right)
\end{align}
for the pairs of Banach spaces
\begin{align*}
(X,Y)=(\tilde{L}^\infty_0(B^s_{2,1}), \tilde{L}^2_v(B^s_{2,1}))\ \mathrm{and}\ (\tilde{L}^\infty_{\beta +1} (B^s_{2,1}),\tilde{L}^\infty_\beta (B^s_{2,1})),\ \beta\ge 0.
\end{align*} 
For the first pair, one has 
\begin{align*}
|\!|\!| e^{tA} f |\!|\!|_{\alpha, \tilde{L}^\infty_0 (B^s_{2,1})} = \sup_{t\ge 0}(1+t)^\alpha \Vert e^{tA}f\Vert_{\tilde{L}^\infty_0 (B^s_{2,1})}\le \max_{t\ge 0}(1+t)^\alpha e^{-\nu_0 t} \cdot \Vert f\Vert_{\tilde{L}^\infty_0 (B^s_{2,1})},
\end{align*}
and
\begin{align*}
\Big|\!\Big|\!\Big| \int^t_0 e^{(t-s)A}Ke^{sB}fds  \Big|\!\Big|\!\Big|_{\alpha, \tilde{L}^\infty_0 (B^s_{2,1})}&\le \sup_{t\ge 0} (1+t)^\alpha \int^t_0 \Vert e^{(t-s)A}Ke^{sB}f \Vert_{\tilde{L}^\infty_0 (B^s_{2,1})}ds\\
&\le \sup_{t\ge 0} (1+t)^\alpha \int^t_0 e^{-\nu_0(t-s)}\Vert Ke^{sB}f \Vert_{\tilde{L}^\infty_0 (B^s_{2,1})}ds\\
&\le \sup_{t\ge 0} (1+t)^\alpha \int^t_0 e^{-\nu_0(t-s)}\Vert e^{sB}f \Vert_{\tilde{L}^2_v(B^s_{2,1})}ds\\
&
\le |\!|\!| e^{tB}f  |\!|\!|_{\alpha, \tilde{L}^2_v(B^s_{2,1})}\sup_{t\ge 0} (1+t)^\alpha \int^t_0 e^{-\nu_0(t-s)}(1+s)^{-\alpha}ds\\
&
\le C |\!|\!| e^{tB}f  |\!|\!|_{\alpha, \tilde{L}^2_v(B^s_{2,1})}.
\end{align*}
Here, to show the second estimate, we have used the fact that $K\in \mathscr{B}(L^2, L^\infty_0)$ by Lemma \ref{lem: Ukai Prop 4.3.1} for the third line and
\begin{align*}
\int^t_0 e^{-\nu_0(t-s)}(1+s)^{-\alpha'}ds \le C(1+t)^{-\alpha'}\quad (\alpha'\ge 0)
\end{align*}
for the last line.

For the second pair, the estimate of $e^{tA}f$ is the same as above, so one has 
\begin{align*}
\Big|\!\Big|\!\Big| \int^t_0 e^{(t-s)A}Ke^{sB}f ds \Big|\!\Big|\!&\Big|_{\alpha, \tilde{L}^\infty_{\beta +1} (B^s_{2,1})}
\le \sup_{t\ge 0}(1+t)^\alpha \overline{\sum} \sup_{v}\langle v\rangle^{\beta +1}\int^t_0 \Vert \Delta_j e^{(t-s)A}Ke^{sB} f \Vert_{L^2_x}ds\\
&\le \sup_{t\ge 0}(1+t)^\alpha \overline{\sum} \sup_{v}\langle v\rangle^{\beta +1}\int^t_0 e^{-\nu(v)(t-s)} \Vert \Delta_j K e^{sB} f\Vert_{L^2_x} ds\\
&\le C\sup_{t\ge 0}(1+t)^\alpha \overline{\sum} \sup_{v}\langle v\rangle^{\beta +1}\int^t_0 e^{-\nu(v)(t-s)} \langle v\rangle^{-\beta-1}\Vert \Delta_j e^{sB}f\Vert_{L^\infty_\beta L^2_x}ds\\
&\le C |\!|\!| e^{tB} f |\!|\!|_{\alpha, \tilde{L}^\infty_\beta (B^s_{2,1})}.
\end{align*}
Thus \eqref{ineq:X-Y estimate} is true for both of the pairs.

Also, since one can show $K\in \mathscr{B}(L^\infty_0, L^\infty_{\beta'})$ with $0\le\beta' \le 1$ by the same method as for proving $K\in \mathscr{B}(L^\infty_\beta, L^\infty_{\beta+1})$, \eqref{ineq:X-Y estimate} is still true for the choice of
\begin{align*}
(X,Y)=(\tilde{L}^\infty_{\beta'} (B^s_{2,1}),\tilde{L}^\infty_0 (B^s_{2,1})).
\end{align*}

Finally, an iterative use of \eqref{ineq:X-Y estimate} gives
\begin{align*}
|\!|\!| e^{tB} f |\!|\!|_{\alpha, \tilde{L}^\infty_\beta (B^s_{2,1})}\le C \left(\Vert f\Vert_{\tilde{L}^\infty_\beta (B^s_{2,1})} +|\!|\!| e^{tB} f |\!|\!|_{\alpha, \tilde{L}^2_v(B^s_{2,1})}\right), \quad \beta \ge 0.
\end{align*}
Now, by applying the first estimate of Lemma \ref{lem:time-decay of semigroup} to the second term of the right-hand side, we derive the desired estimate for $m=0$. When $m=1$, the same proof works as well, and details are omitted for brevity. This completes the proof of Lemma \ref{lem:Time-dependent estimate of semigroup}.
\end{proof}

Basing on Lemma \ref{lem:Time-dependent estimate of semigroup}, we further have

\begin{lem}\label{lem:Estimate of Psi}
Let $0\le \tilde{\alpha}\neq 1$, $s\in\mathbb{R}$, $\beta \ge 0$, and
\begin{align}\label{a.con}
0\le \alpha \le 
\begin{cases}
\min(3/4, \tilde{\alpha}-1/4) \quad&\mathrm{if}\  d=1,\\
\min(1_-,
\tilde{\alpha}) \quad&\mathrm{if}\  d=2,\\
\min(d/4+1/2, \tilde{\alpha}) \quad&\mathrm{if}\  d\ge 3,
\end{cases}
\end{align}
where $1_-$ denoted $1-\delta$ for an arbitrary small constant $\delta>0$.
Then it holds that
\begin{align*}
|\!|\!| \Psi f |\!|\!|_{\alpha, \tilde{L}^\infty_\beta (B^s_{2,1})} \le C\left(|\!|\!| f |\!|\!|_{\tilde{\alpha}, \tilde{L}^\infty_\beta (B^s_{2,1})}+|\!|\!| \nu f |\!|\!|_{\tilde{\alpha}, \tilde{L}^2_v(B^s_{2,1})\cap L^2_v L^1_x}\right),
\end{align*}
where
\begin{align*}
(\Psi f) (t) :=\int^t_0 e^{(t-s)B}(I-P_0)\nu f(s)ds.
\end{align*}
\end{lem}

\begin{proof}
Set
\begin{align*}
(\Psi_n f)(t)=\int^t_0 e^{(t-s)A}(I-P_0)^n\nu f(s)ds, \quad n=0, 1.
\end{align*}
We first observe that
\begin{align}\label{ineq: Time-Dependent Besov Series}
|\!|\!|  f |\!|\!|_{\alpha, \tilde{L}^\infty_\beta (B^s_{2,1})}\le M \Rightarrow \Vert \Delta_j f(t,\cdot,v)\Vert_{L^2_x} \le M(1+t)^{-\alpha} 2^{-js} c_j \langle v\rangle^{-\beta}
\end{align}
for some $c_j\in \ell^1$ with $\Vert c_j\Vert_{\ell^1}\le 1$. Here note that $c_j$ can be independent of $v$ and $t$, for instance, one can take
\begin{align*}
c_j=\frac{2^{js}}{M}\sup_{t,v} (1+t)^\alpha \langle v\rangle^\beta \Vert \Delta_j f(t,v)\Vert_{L^2_x}.
\end{align*}
Then we have
\begin{align*}
|\!|\!| \Psi_0 f |\!|\!|_{\alpha, \tilde{L}^\infty_\beta (B^s_{2,1})}&\le \sup_{t\ge 0}(1+t)^\alpha \overline{\sum}\sup_v \langle v\rangle^\beta \int^t_0 e^{-\nu(v)(t-s)}\nu(v) \Vert \Delta_j f\Vert_{L^2_x}ds \\
&\le |\!|\!| f |\!|\!|_{\tilde{\alpha}, \tilde{L}^\infty_\beta (B^s_{2,1})} \sup_{t\ge 0}(1+t)^\alpha \sum_j c_j \sup_v  \int^t_0 e^{-\nu(v)(t-s)}\nu(v) (1+s)^{-\tilde{\alpha}}ds \\
&\le |\!|\!| f |\!|\!|_{\tilde{\alpha}, \tilde{L}^\infty_\beta (B^s_{2,1})}\sup_{t\ge 0}(1+t)^{\alpha-\tilde{\alpha}}\le |\!|\!| f |\!|\!|_{\tilde{\alpha}, \tilde{L}^\infty_\beta (B^s_{2,1})},
\end{align*}
where we have used the fact that 
\begin{align*}
\sup_v  \int^t_0 e^{-\nu(v)(t-s)}\nu(v) (1+s)^{-\tilde{\alpha}}ds \le C(1+t)^{-\tilde{\alpha}}.
\end{align*}

To estimate $|\!|\!| \Psi_1 f |\!|\!|_{\alpha, L^\infty_\beta B^s_{2,1}}$, we recall the fact that $P_0g=\sum_{i=0}^{d+1}(\chi_i, g)\chi_i$, where
\begin{align*}
\chi_0(v)=M^{1/2}, \ \chi_i(v)=v_iM^{1/2}\  (i=1,\cdots,d),\ \chi_{d+1}(v)=\vert v\vert^2M^{1/2}.
\end{align*}
Thus for each $i$, one has
\begin{align*}
\left| (\nu\Vert f(s,\cdot, v)\Vert_{B^s_{2,1}}, \chi_i)\right| \le C \Vert f(s,\cdot,\cdot)\Vert_{L^\infty_\beta( B^s_{2,1})}
\end{align*}
and
\begin{align*}
\sup_{v\in \mathbb{R}^d}\langle v\rangle^\beta \int^t_0 e^{-\nu(v)(t-s)}(1+s)^{-\tilde{\alpha}}\vert \chi_i(v)\vert ds\le C\sup_{v\in\mathbb{R}^d}\langle v\rangle^\beta \frac{(1+t)^{-\tilde{\alpha}}}{\nu(v)}M^{1/2}(v)\le C(1+t)^{-\tilde{\alpha}}.
\end{align*}
We divide $\Psi_1 f$ by the difference and estimate each term as follows.
First of all, it holds that 
\begin{align*}
\Big|\!\Big|\!\Big| \int^t_0 e^{(t-s)A}P_0(\nu f)ds &\Big|\!\Big|\!\Big|_{\alpha, \tilde{L}^\infty_\beta (B^s_{2,1})} \le \sup_{t\ge 0}(1+t)^\alpha \overline{\sum} \sup_{v} \langle v\rangle^\beta  \int^t_0 \Vert \Delta_j e^{(t-s)A}P_0(\nu f)\Vert_{L^2_x} ds\\
&\le \sup_{t\ge 0}(1+t)^\alpha \overline{\sum} \sup_{v} \langle v\rangle^\beta  \int^t_0 e^{-\nu (v)(t-s)}\Vert \Delta_j P_0(\nu f)\Vert_{L^2_x}ds\\
&\le \sup_{t\ge 0}(1+t)^\alpha \overline{\sum} \sup_{v} \langle v\rangle^\beta  \int^t_0 e^{-\nu (v)(t-s)} \sum_{i=0}^{d+1} \vert (\nu \Vert \Delta_j f\Vert_{L^2_x}, \chi_i)\vert \vert \chi_i(v)\vert ds\\
&\le C\sup_{t\ge 0}(1+t)^\alpha \overline{\sum} \sup_{v} \langle v\rangle^\beta  \int^t_0 e^{-\nu (v)(t-s)} \Vert f\Vert_{\tilde{L}^\infty_\beta (B^s_{2,1})}c_j 2^{-js} \sum_{i=0}^{d+1}\vert \chi_i(v)\vert ds\\
&\le C|\!|\!| f |\!|\!|_{\tilde{\alpha}, \tilde{L}^\infty_\beta (B^s_{2,1})}.
\end{align*}
Here we have used $\alpha \le \tilde{\alpha}$ due to \eqref{a.con}, and also we emphasize that $c_j \in \ell^1$ does not depend on the time variable by the same reason of \eqref{ineq: Time-Dependent Besov Series}. Thus one can deduce that  
\begin{align*}
|\!|\!| \Psi_n f |\!|\!|_{\alpha, \tilde{L}^\infty_\beta (B^s_{2,1})}\le C|\!|\!| f |\!|\!|_{\tilde{\alpha}, \tilde{L}^\infty_\beta (B^s_{2,1})}, \quad n=0,1.
\end{align*}

To further proceed the proof, for the time being we claim that $\Psi = \Psi_1+\Psi_0 \nu^{-1}K\Psi$. Since it is clear to see 
$(\Psi f)(0)=(\Psi_n f)(0)=0$ with $n=0$ and $1$, it suffices to show that the derivatives in time on both sides are identical.
Putting $G=(I-P_0)\nu f$ for brevity, one has 
\begin{align*}
\frac{d}{dt}(\Psi f)(t)&= G(t)+\int^t_0 e^{(t-s)B}BG(s)ds\\
&=G(t)+\int^t_0 B \Big[ e^{(t-s)A}+\int^{(t-s)}_0e^{(t-s-\tau)A}K e^{\tau B} d\tau \Big] G(s)ds\\
&=G(t)+\int^t_0 A e^{(t-s)A}G(s)ds+\int^t_0 K e^{(t-s)A}G(s)ds\\
&+\int^t_0 B \int^{t-s}_0e^{(t-s-\tau)A}Ke^{\tau B} d\tau G(s)ds,
\end{align*}
in terms of \eqref{eqn:operator identity}. Here, sum of the first two terms is identical to $d(\Psi_1 f)/dt$, and sum of the other terms is given by
\begin{align*}
&\int^t_0 K \Big[ e^{(t-s)B}-\int^{(t-s)}_0 e^{(t-s-\tau)A}Ke^{\tau B}d\tau\Big] G(s)ds+\int^t_0 B \int^{(t-s)}_0e^{(t-s-\tau)A}Ke^{\tau B} d\tau F(s)ds\\
&=\int^t_0 K e^{(t-s)B}F(s)ds +\int^t_0 A \int^{(t-s)}_0e^{(t-s-\tau)A}Ke^{\tau B} d\tau F(s)ds\\
&=\int^t_0 K e^{(t-s)B}F(s)ds +\int^t_0 A \int^{t}_{s}e^{(t-\tau')A}Ke^{(\tau'-s) B} d\tau' F(s)ds\\
&=\int^t_0 K e^{(t-s)B}F(s)ds +\int^t_0 \int^{\tau'}_{0}A e^{(t-\tau')A}Ke^{(\tau'-s) B} F(s)ds d\tau', 
\end{align*}
which corresponds to $d(\Psi_0 \nu^{-1}K\Psi f)/dt$. This then proves the claim. 

By the resulting identity $\Psi = \Psi_1+\Psi_0 \nu^{-1}K\Psi$, one can proceed as in the proof of \eqref{ineq:X-Y estimate} to obtain 
\begin{align*}
|\!|\!| \Psi f |\!|\!|_{\alpha, \tilde{L}^\infty_\beta (B^s_{2,1})} \le C\left(|\!|\!| \Psi_1 f |\!|\!|_{\alpha, \tilde{L}^\infty_\beta (B^s_{2,1})} + |\!|\!| \Psi f |\!|\!|_{\alpha, \tilde{L}^2_v(B^s_{2,1})}\right).
\end{align*}
Note that it is straightforward to estimate the first term on the right-hand due to the extra time-decay, but the estimate of the second term depends on the spatial dimension $d$. In fact, to estimate the second term on the right, it follows from Lemma \ref{lem:time-decay of semigroup} that 
\begin{align*}
|\!|\!| \Psi f |\!|\!|_{\alpha, \tilde{L}^2_v(B^s_{2,1})} &\le C\sup_{t\ge 0} (1+t)^\alpha \int^t_0 (1+t-s)^{-d/4-1/2}\Vert \nu f\Vert_{\tilde{L}^2_v(B^s_{2,1})\cap L^2_vL^1_x}ds.
\end{align*}
Thus, recalling \eqref{a.con}, it remains to verify that  
\begin{align}
\label{a.ti}
\sup_{t\ge 0}\, (1+t)^\alpha \int^t_0 (1+t-s)^{-d/4-1/2}(1+s)^{-\tilde{\alpha}}ds
\end{align}
is finite.
If $d\ge 3$, the time integral above is  bounded by $C(1+t)^{-\min(d/4+1/2, \tilde{\alpha})}$. If $d=2$, then one has $d/4+1/2=1$, and   
thus the bound of the time integral can be taken as
\begin{equation*}
C[(1+t)^{-1}\log (1+t) + (1+t)^{-\tilde{\alpha}}]\le C(1+t)^{-\min(1-\delta,\tilde{\alpha})}.
\end{equation*} 
If $d=1$, the bound is given by $C(1+t)^{-\min(3/4, \tilde{\alpha}-1/4)}$. Collecting all cases, for any $d\geq 1$ we have proved that \eqref{a.ti} is finite. 
This then completes the proof of Lemma \ref{lem:Estimate of Psi}. 
\end{proof}

We are now devoted to obtaining the nonlinear estimate, which is crucial to apply the Banach fixed point theorem.

\begin{lem}\label{lem:bilinear estimate of Gamma}
Assume $s>0$ and $\beta >\gamma + d/2$. Then it holds that 
\begin{align*}
\sum_{j \ge -1} 2^{js} \Vert \Delta_j \Gamma(f,g)\Vert_{L^2_v L^2_x}  \le C(\Vert f\Vert_{\tilde{L}^\infty_\beta (B^s_{2,1})} \Vert g\Vert_{L^\infty_\beta L^\infty_x}+\Vert f\Vert_{L^\infty_\beta L^\infty_x} \Vert g\Vert_{\tilde{L}^\infty_\beta (B^s_{2,1})}).
\end{align*}
\end{lem}

\begin{proof}
Applying the Bony decomposition to the product of $f$ and $g$ in $\Gamma$, we divide $\Gamma(f,g)$ into
\begin{align*}
\sum_{k=1}^3\Gamma^k(f,g)&=\sum_{k=1}^3\Gamma^k_{gain}(f,g)-\Gamma^k_{loss}(f,g), \quad
&\Gamma^1(f,g)=\sum_{i}\Gamma(\Delta_i f, S_{i-1}g), \\
\Gamma^2(f,g)&=\sum_{i}\Gamma(S_{i-1} f, \Delta_i g),\quad
&\Gamma^3(f,g)=\sum_{i,i':\vert i-i'\vert\le 1}\Gamma(\Delta_i f, \Delta_{i'} g),
\end{align*}
where $\Gamma^k_{gain}$ and $\Gamma^k_{loss}$ are defined according to the conventional decomposition of $Q$ into the gain term and loss term, respectively. We first give bounds to the loss terms. Recall the fact that 
\begin{align*}
\Delta_j \sum_i (\Delta_i f S_{i-1}g)=\sum_{\vert i-j\vert \le 4} \Delta_j (\Delta_i f S_{i-1}g)
\end{align*}
due to \eqref{Nearly Orthogonal}. It is well-known that one has  $\Vert \Delta_j g \Vert_{L^p}\le C_p \Vert g\Vert_{L^p}$ and $\Vert S_j g \Vert_{L^p}\le C_p \Vert g\Vert_{L^p}$ for any $j \ge -1$. For $f \in \tilde{L}^\infty_\beta (B^s_{2,1})$, there is a summable positive sequence $\{c_j\}$ such that for any $j$,
\begin{align}\label{ineq: Pointwise estimate of cut function}
\Vert \Delta_j f(\cdot,v)\Vert_{L^2_x}\le \Vert f\Vert_{\tilde{L}^\infty_\beta (B^s_{2,1})}c_j 2^{-js}\langle v\rangle^{-\beta}.
\end{align}
Note that $c_j=c_j^f$ should depend on the function $f$. However, for brevity we would not point out such dependence here and in the sequel.

We first estimate $\Gamma^1_{loss}(f,g)$ as
\begin{align*}
&\overline{\sum} \Vert \Delta_j \Gamma^1_{loss}(f,g)\Vert_{L^2_v L^2_x} \\
&\le \overline{\sum} \sum_{\vert i-j\vert \le 4} \Big( \int_{\mathbb{R}^d} \int_{\mathbb{R}^d} \Big\vert \int_{\mathbb{R}^d} \int_{\mathbb{S}^{d-1}}\vert v-v_* \vert^\gamma  b_0(\theta) M_*^{1/2}\Delta_j (\Delta_i f_* S_{i-1}g) dv_*d\omega \Big\vert^2 dvdx\Big)^{1/2}\\
&\le C\overline{\sum} \sum_{\vert i-j\vert \le 4} \Big( \int_{\mathbb{R}^d} \int_{\mathbb{R}^d} \Big\vert \int_{\mathbb{R}^d} \vert v-v_* \vert^\gamma  M_*^{1/2} \Delta_j (\Delta_i f_* S_{i-1}g) dv_* \Big\vert^2 dvdx\Big)^{1/2}\\
&\le C\overline{\sum} \sum_{\vert i-j\vert \le 4} \Big( \int_{\mathbb{R}^d} \Big\vert \int_{\mathbb{R}^d}  \Big(\int_{\mathbb{R}^d}  \vert \Delta_j (\Delta_i f_* S_{i-1}g)\vert^2dx\Big)^{1/2} \vert v-v_* \vert^\gamma  M_*^{1/2} dv_* \Big\vert^2 dv\Big)^{1/2}\\
&\le C\overline{\sum} \sum_{\vert i-j\vert \le 4} \Big( \int_{\mathbb{R}^d} \Big\vert \int_{\mathbb{R}^d} \Vert \Delta_i f_*\Vert_{L^2_x} \Vert g\Vert_{L^\infty_x} \vert v-v_* \vert^\gamma  M_*^{1/2} dv_* \Big\vert^2 dv\Big)^{1/2}\\
&\le C\Vert f\Vert_{\tilde{L}^\infty_\beta (B^s_{2,1})}\Vert g\Vert_{L^\infty_\beta L^\infty_x}\sum_{j\ge -1} 2^{(j-i)s} \sum_{\vert i-j\vert \le 4}c_i\Big( \int_{\mathbb{R}^d} \Big\vert \int_{\mathbb{R}^d} \langle v \rangle^{-\beta} \langle v_* \rangle^{-\beta} \vert v-v_* \vert^\gamma  M_*^{1/2} dv_* \Big\vert^2 dv\Big)^{1/2}\\
&\le C\Vert f\Vert_{\tilde{L}^\infty_\beta (B^s_{2,1})}\Vert g\Vert_{L^\infty_\beta L^\infty_x}\sum_{j\ge -1} 2^{(j-i)s} \sum_{\vert i-j\vert \le 4}c_i\Big(\int_{\mathbb{R}^d} \langle v \rangle^{-2\beta}  \langle v \rangle^{2\gamma}dv\Big)^{1/2}.
\end{align*}
Here we have used the integrability of $b$ on $\mathbb{S}^{d-1}$, the Minkowski integral inequality $\Vert \Vert\cdot\Vert_{L^1_{v_*}}\Vert_{L^2_x}\le \Vert \Vert\cdot\Vert_{L^2_x}\Vert_{L^1_{v_*}}$, $L^2$-boundedness of $\Delta_j$ and $S_{j-1}$, \eqref{ineq: Pointwise estimate of cut function}, and Lemma \ref{equiv: v_* Integral}. 
The last integral is bounded by the assumption $\beta >\gamma + d/2$, and the sum is finite because of the discrete Young's inequality and positivity of $s$.  Thus the estimate of $\Gamma^1_{loss}(f,g)$ is proved.
Due to symmetry, it also holds that 
\begin{align*}
\overline{\sum} \Vert \Delta_j \Gamma^2_{loss}(f,g)\Vert_{L^2_v L^2_x}\le C \Vert f\Vert_{L^\infty_\beta L^\infty_x}\Vert g\Vert_{\tilde{L}^\infty_\beta (B^s_{2,1})}.
\end{align*}

In order to estimate the term coming from $\Gamma^3_{loss}(f,g)$, we recall the following property:
\begin{align*}
\Delta_j\sum_{i,i':\vert i-i'\vert\le1}\Delta_i f \Delta_{i'} g = \sum_{\max(i,i')\ge j-2}\sum_{\vert i-i'\vert\le1}\Delta_j(\Delta_i f \Delta_{i'} g).
\end{align*}
For brevity we write
\begin{align*}
\overline{\sum}'=\sum_{j\ge -1} 2^{js}\sum_{\max(i,i')\ge j-2}\sum_{\vert i-i'\vert\le1}.
\end{align*}
Then it follows that
\begin{align*}
&\overline{\sum} \Vert \Delta_j \Gamma^3_{loss}(f,g)\Vert_{L^2_v L^2_x} \\
&\le C\overline{\sum}' \Big( \int_{\mathbb{R}^d} \int_{\mathbb{R}^d} \Big\vert \int_{\mathbb{R}^d} \int_{\mathbb{S}^{d-1}}\vert v-v_* \vert^\gamma  b_0(\theta) M_*^{1/2} \Delta_j (\Delta_i f_* \Delta_{i'}g) dv_*d\omega \Big\vert^2 dvdx\Big)^{1/2}\\
&\le C\overline{\sum}' \Big(\int_{\mathbb{R}^d} \Big\vert\int_{\mathbb{R}^d}\Big(  \int_{\mathbb{R}^d} \vert\Delta_j (\Delta_i f_* \Delta_{i'}g)\vert^2 dx\Big)^{1/2}\vert v-v_* \vert^\gamma  M_*^{1/2} dv_* \Big\vert^2 dv\Big)^{1/2}\\
&\le C\overline{\sum}' \Big(\int_{\mathbb{R}^d} \Big\vert\int_{\mathbb{R}^d}\Vert\Delta_i f_*\Vert_{L^2_x} \Vert g\Vert_{L^\infty_x} \vert v-v_* \vert^\gamma  M_*^{1/2} dv_* \Big\vert^2 dv\Big)^{1/2}\\
&\le C\Vert f\Vert_{\tilde{L}^\infty_\beta (B^s_{2,1})} \Vert g\Vert_{L^\infty_\beta L^\infty_x}\sum_{j\ge -1} \sum_{i\ge j-3}2^{(j-i)s}c_i \Big(\int_{\mathbb{R}^d} \Big\vert\int_{\mathbb{R}^d} \langle v_*\rangle^{-\beta} \langle v\rangle^{-\beta} \vert v-v_* \vert^\gamma  M_*^{1/2} dv_* \Big\vert^2 dv\Big)^{1/2}\\
&\le C\Vert f\Vert_{\tilde{L}^\infty_\beta (B^s_{2,1})} \Vert g\Vert_{L^\infty_\beta L^\infty_x}\sum_{i\ge -4} \sum_{j=-1}^{i+3}2^{(j-i)s}c_i.
\end{align*}
The sum in the last line is further bounded from the same reason as used before. Thus, we obtain the desired estimates on all the loss terms.

The gain terms $\Gamma^k_{gain}(f,g)$ can be  estimated  as for 
$\Gamma^k_{loss}(f,g)$. 
Indeed, it suffices to consider the boundedness of 
\begin{align*}
\int_{\mathbb{R}^d} \left(\int_{\mathbb{R}^d}\int_{\mathbb{S}^{d-1}} \langle v'_*\rangle^{-\beta} \langle v'\rangle^{-\beta} \vert v-v_* \vert^\gamma b_0(\theta)  M_*^{1/2} dv_* d\omega\right)^2 dv,
\end{align*}
where we have applied the inequality $\Vert \Vert\cdot\Vert_{L^1_{v_*, \omega}}\Vert_{L^2_x}\le \Vert \Vert\cdot\Vert_{L^2_x}\Vert_{L^1_{v_*, \omega}}$. The above integral is finite, because the conservation law of energy yields
\begin{align}\label{ineq: Pre-Post Bessel Potential}
\langle v'_*\rangle \langle v'\rangle =\left[(1+\vert v'_*\vert^2)((1+\vert v'\vert^2)\right]^{1/2} \ge (1+\vert v'_*\vert^2+\vert v'\vert^2)^{1/2}\ge \langle v\rangle,
\end{align}
so that we are able to apply Lemma \ref{equiv: v_* Integral} once again. Therefore, by combining all  estimates, we complete the proof of Lemma \ref{lem:bilinear estimate of Gamma}.
\end{proof}

We point out that the estimate of the nonlinear term for the case of soft potentials can also be derived by the similar argument above in spite of the fourth remark in Subsection \ref{sub1.3}; see the proof of Theorem \ref{thm: Time-Decay of the nonlinear solution for soft potential}.

The continuous embedding $B^{d/2}_{2,1}(\mathbb{R}^d)\hookrightarrow L^\infty(\mathbb{R}^d)$ leads to the following

\begin{cor}\label{cor:CL estimate of Gamma}
Assume $s\ge d/2$ and $\beta>\gamma +d/2$. Then it holds
\begin{align*}
\Vert \Gamma(f,g)\Vert_{\tilde{L}^2_v(B^s_{2,1})}=\sum_{j \ge -1} 2^{js} \Vert \Delta_j \Gamma(f,g)\Vert_{L^2_v L^2_x}  \le C\Vert f\Vert_{\tilde{L}^\infty_\beta (B^s_{2,1})}\Vert g\Vert_{\tilde{L}^\infty_\beta (B^s_{2,1})}.
\end{align*}
\end{cor}

We are now ready to show the global existence of a mild solution to the Cauchy problem \eqref{eqn: BE near M}.

\begin{proof}[Proof of Theorem \ref{thm: Solution for the hard potential case}]

It suffices to show
\begin{align}\label{ineq:Contraction property}
|\!|\!| \nu^{-1}\Gamma(f,g) |\!|\!|_{2\alpha, \tilde{L}^\infty_\beta (B^s_{2,1})}+|\!|\!| \Gamma(f,g) |\!|\!|_{2\alpha, \tilde{L}^2_v (B^s_{2,1})\cap L^2_v L^1_x}\le C|\!|\!| f |\!|\!|_{\alpha, \tilde{L}^\infty_\beta (B^s_{2,1})}|\!|\!| g |\!|\!|_{\alpha, \tilde{L}^\infty_\beta (B^s_{2,1})}.
\end{align}
Indeed, let us first suppose that the above estimate is true. Since the mild form of the Cauchy problem can also be written as
\begin{align*}
f(t)=e^{tB}f_0+(\Psi \nu^{-1} \Gamma(f,f))(t)=:N(f)(t),
\end{align*}
we obtain a unique global mild solution if the nonlinear mapping $N$ is a contraction for initial data $f_0$ sufficiently small in some sense. Together with \eqref{ineq:Contraction property}, we apply Lemma \ref{lem:Estimate of Psi} with $2\alpha = \tilde{\alpha}=d/2(1/q-1/2)$, to deduce that 
\begin{align*}
|\!|\!| N(f) |\!|\!|_{\alpha, \tilde{L}^\infty_\beta (B^s_{2,1})} &\le C\Vert f_0\Vert_{\tilde{L}^\infty_\beta (B^s_{2,1}) \cap L^2_v L^q_x}+C|\!|\!| f|\!|\!|_{\alpha, \tilde{L}^\infty_\beta (B^s_{2,1})}^2,\\
|\!|\!| N(f)-N(\tilde{f}) |\!|\!|_{\alpha, \tilde{L}^\infty_\beta (B^s_{2,1})}&\le C(|\!|\!| f |\!|\!|_{\alpha, \tilde{L}^\infty_\beta (B^s_{2,1})}+|\!|\!| \tilde{f} |\!|\!|_{\alpha, \tilde{L}^\infty_\beta (B^s_{2,1})})|\!|\!| f-\tilde{f} |\!|\!|_{\alpha, \tilde{L}^\infty_\beta (B^s_{2,1})},
\end{align*}
where we have used the inclusion $\tilde{L}^\infty_\beta (B^s_{2,1})\hookrightarrow \tilde{L}^2_v(B^s_{2,1})$ with $\beta >d/2$.
Also, we remark that even if $2\alpha=\tilde{\alpha}=1$, which may cause a logarithmic increase, the argument for Lemma \ref{lem:Estimate of Psi} provides the sufficient time-decay rate for proceeding the estimates, and details of the proof are omitted for brevity. Then it follows that $N$ is a contraction on the set 
\begin{align*}
\left\{ \left.f\in L^\infty\left(0,\infty; \tilde{L}^\infty_\beta (B^s_{2,1})\right)\ \right| \ |\!|\!| f |\!|\!|_{\alpha, \tilde{L}^\infty_\beta (B^s_{2,1})}\le a\right\}
\end{align*}
for a suitable constant $a>0$, provided that $f_0$ is small enough in the space $\tilde{L}^\infty_\beta (B^s_{2,1}) \cap L^2_v L^q_x$.\qed
\medskip

Now it remains to show the nonlinear estimate \eqref{ineq:Contraction property}. We start with the bilinear estimate of $\nu^{-1}\Gamma(\cdot,\cdot)$ in $L^\infty_\beta$ with respect to velocity variable only, cf.~\cite{UA}.

\begin{lem}
For $\beta \ge 0$, it holds that 
\begin{align}\label{lem:Estimate of nu inverse gamma}
\Vert \nu^{-1}\Gamma(F,G)\Vert_{L^\infty_\beta}\le C\Vert F\Vert_{L^\infty_\beta}\Vert G\Vert_{L^\infty_\beta}.
\end{align}
\end{lem}
First of all, one has 
\begin{align*}
|\!|\!| \Gamma(f,g) |\!|\!|_{2\alpha, \tilde{L}^2_v (B^s_{2,1})}\le C |\!|\!| f |\!|\!|_{\alpha, \tilde{L}^\infty_\beta (B^s_{2,1})}|\!|\!| g |\!|\!|_{\alpha, \tilde{L}^\infty_\beta (B^s_{2,1})},
\end{align*}
which is an immediate consequence of Corollary \ref{cor:CL estimate of Gamma}, and also it holds that 
\begin{align*}
|\!|\!| \Gamma(f,g) |\!|\!|_{2\alpha, L^2_vL^1_x} &\le\sup_{t\ge 0}(1+t)^{2\alpha} \Vert \Gamma(\Vert f\Vert_{L^2_x},\Vert g\Vert_{L^2_x})\Vert_{L^2_v}\\
&\le \sup_{t\ge 0} (1+t)^{2\alpha}\Vert \nu^{-1} \Gamma(\Vert f\Vert_{L^2_x},\Vert g\Vert_{L^2_x})\Vert_{L^\infty_\beta}\Vert \nu \langle\cdot\rangle^{-\beta}\Vert_{L^2}\\
&\le C|\!|\!| f |\!|\!|_{\alpha, L^\infty_\beta B^s_{2,1}}|\!|\!| g |\!|\!|_{\alpha, L^\infty_\beta B^s_{2,1}}\\
&\le C|\!|\!| f |\!|\!|_{\alpha, \tilde{L}^\infty_\beta (B^s_{2,1})}|\!|\!| g |\!|\!|_{\alpha, \tilde{L}^\infty_\beta (B^s_{2,1})},
\end{align*}
in terms  of \eqref{lem:Estimate of nu inverse gamma} and the boundedness of $\Vert \nu \langle\cdot\rangle^{-\beta}\Vert_{L^2}$ for $\beta>\gamma+d/2$.

The remaining part is to further show
\begin{align*}
|\!|\!| \nu^{-1}\Gamma(f,g) |\!|\!|_{2\alpha, \tilde{L}^\infty_\beta (B^s_{2,1})}\le C|\!|\!| f |\!|\!|_{\alpha, \tilde{L}^\infty_\beta (B^s_{2,1})}|\!|\!| g |\!|\!|_{\alpha, \tilde{L}^\infty_\beta (B^s_{2,1})}.
\end{align*}
Note that although \eqref{lem:Estimate of nu inverse gamma} cannot be directly applied, one can still proceed as in the proof of Lemma \ref{lem:bilinear estimate of Gamma}. We only consider the estimate of the term containing $\Gamma^3_{loss}$, since the other terms can be similarly estimated. Indeed, one has 
\begin{align*}
&\overline{\sum}\sup_v \langle v \rangle^\beta \Vert \Delta_j \nu^{-1}\Gamma^3_{loss}(f,g)\Vert_{L^2_x} \\
&\le C\overline{\sum}' \sup_v \langle v \rangle^\beta \Big( \int_{\mathbb{R}^d} \Big( \int_{\mathbb{R}^d}  \vert v-v_*\vert^\gamma M_*^{1/2} \nu^{-1} \Delta_j (\Delta_i f_* \Delta_{i'} g) dv_* \Big)^2 dx\Big)^{1/2}\\
&\le C\overline{\sum}' \sup_v \langle v \rangle^\beta \int_{\mathbb{R}^d} \Big(\int_{\mathbb{R}^d} \vert \Delta_j (\Delta_i f_* \Delta_{i'} g)\vert^2 dx\Big)^{1/2} \vert v-v_*\vert^\gamma \nu^{-1} M_*^{1/2}dv_*\\
&\le C\overline{\sum}' \sup_v \langle v \rangle^\beta \int_{\mathbb{R}^d} \Vert \Delta_i f_*\Vert_{L^2_x} \Vert g\Vert_{L^\infty_x} \vert v-v_*\vert^\gamma\nu^{-1} M_*^{1/2}dv_*\\
&\le C\Vert f\Vert_{\tilde{L}^\infty_\beta (B^s_{2,1})}\Vert g\Vert_{L^\infty_\beta L^\infty_x}\sum_{i\ge -4} \sum_{j=-1}^{i+3}2^{(j-i)s}c_i  \sup_v \langle v \rangle^\beta \int_{\mathbb{R}^d}  \langle v_*\rangle^{-\beta} \langle v\rangle^{-\beta} \vert v-v_*\vert^\gamma\langle v\rangle^{-\gamma} M_*^{1/2}dv_*.
\end{align*}
Here, owing to Lemma \ref{equiv: v_* Integral}, the term containing $\sup_v$ in the last line is dominated by a constant.  
Thus we have the desired estimate. As we have shown \eqref{ineq:Contraction property}, the Banach fixed point theorem assures the existence of a unique global mild solution. This completes the proof of Theorem \ref{thm: Solution for the hard potential case}.
\end{proof}

\section{Solution for the soft potential case}\label{Ch: Soft}

We now turn to the proof of Theorem \ref{thm: Solution for the soft potential case} in the case of soft potentials  $-d<\gamma<0$. Through this section, we set the spatial dimension $d\geq 3$. As in the hard potential case, we start from considering the time-decay in the space $\tilde{L}^2_v(B^s_{2,1})$ for the solution to the Cauchy problem \eqref{eqn: Linearized Equation} with the help of Lemma \ref{lem: DYZ} whose proof is based on the pure energy method. In contrast to Lemma \ref{lem:time-decay of semigroup}, one can not have any extra decay for the non-fluid initial data.   
 

\begin{lem}\label{lem: Time-decay of velocity-weighted square integrable space}
Assume $-d<\gamma<0$. Take $\ell \ge 0$, $1 \le q\le 2$, and $J>d(1/q-1/2)=2\alpha$. Let $f(t,x,v)$ be the solution to the Cauchy problem \eqref{eqn: Linearized Equation} with initial data $f_0(x,v)$. Then it holds that 
\begin{align}
\Vert \nu^{-\ell/2} f(t) \Vert_{\tilde{L}^2_v(B^s_{2,1})} \le &C (1+t)^{-J/2} \Vert \nu^{-(\ell+J_+)/2}f_0\Vert_{\tilde{L}^2_v(B^s_{2,1})}\notag\\
&\qquad\qquad+C(1+t)^{-\alpha}\Vert \nu^{-(\ell+J_+)/2} f_0\Vert_{L^2_vL^q_x},\label{ad.lem.td}
\end{align} 
for all $t\geq 0$.
\end{lem}

\begin{proof}
Recall $\rho(\xi)=|\xi|^2/(1+|\xi|^2)$. By Lemma \ref{lem: DYZ} we have
\begin{align*}
\Vert \nu^{-\ell/2} \Delta_j f(t)\Vert_{L^2_{x.v}}^2 &= \int_{\mathbb{R}^d} \Vert \nu^{-\ell/2} \phi(2^{-j}\xi)\hat{f}(t,\xi)\Vert_{L^2_v}^2 d\xi\\
&\le C  \int_{\mathbb{R}^d} (1+\rho(\xi)t)^{-J} \Vert \nu^{-(\ell+J_+)/2} \phi(2^{-j}\xi) \hat{f}_0(\xi)\Vert_{L^2_v}^2 d\xi\\
&= C \left\{\int_{\vert \xi \vert \ge 1}+\int_{\vert \xi \vert \le 1}\right\} (\cdots) d\xi=:I_1+I_2.
\end{align*}
For $I_1$, we notice $1+\rho(\xi)t\sim 1+t$ on $\{\vert \xi\vert \ge 1\}$. Thus, one has  
\begin{align*}
I_1 \le C(1+t)^{-J} \Vert \nu^{-(\ell+J_+)/2} \Delta_j f_0\Vert_{L^2_{x,v}}^2.
\end{align*}
For $I_2$, we take the triplet $(q, p, p')$ satisfying $1/p+1/p'=1$ and $1/2p+1/q=1$. The H\"{o}lder inequality gives
\begin{align*}
I_2 &\le \Big( \int_{\vert \xi\vert \le 1} (1+\rho(\xi)t)^{-Jp'} d\xi \Big)^{1/p'} \Big( \int_{\vert \xi\vert \le 1} \Vert \nu^{-(\ell+J_+)/2}\phi(2^{-j}\xi)\hat{f}_0(\xi)\Vert_{L^2_v}^{2p} d\xi \Big)^{1/p}\\
&\le C \tilde{\psi}_{Jp'}(t)^{1/p'} \Big( \int_{\vert \xi\vert \le 1} \Vert \nu^{-(\ell+J_+)/2}\phi(2^{-j}\xi)\hat{f}_0(\xi)\Vert_{L^2_v}^{2p} d\xi \Big)^{1/p},
\end{align*}
where by the change of variable $r \rightarrow s=r^2t/(1+r^2)$, it holds that 
\begin{align*}
\tilde{\psi}_{Jp'}(t)&=\int^1_0 \Big(1+\frac{r^2}{1+r^2}t\Big)^{-Jp'}r^{d-1} dr = \int^{t/2}_0 (1+s)^{-Jp'}\Big(\frac{s}{t-s}\Big)^{(d-1)/2} \frac{t}{2\sqrt{s}}(t-s)^{-3/2} ds\\
&\le C(1+t)^{-d/2}  \int^{t/2}_0 (1+s)^{-Jp'}s^{d/2-1}ds  \le C(1+t)^{-d/2},
\end{align*}
due to $J>d(1/q-1/2)=d/2p'$. Therefore, combining estimates on $I_1$ and $I_2$, we have obtained
\begin{align*}
&\Vert \nu^{-\ell/2} \Delta_j f(t)\Vert_{L^2_{x.v}}^2 \le C(1+t)^{-J} \Vert \nu^{-(\ell+J_+)/2} \Delta_j f_0\Vert_{L^2_{x,v}}^2 \\
&\qquad\qquad\qquad\qquad\quad+ C(1+t)^{-2\alpha} \Big( \int_{\vert \xi\vert \le 1} \Vert \nu^{-(\ell+J_+)/2}\phi(2^{-j}\xi)\hat{f}_0(\xi)\Vert_{L^2_v}^{2p} d\xi \Big)^{1/p}.
\end{align*}
Taking the square root of the above inequality, further taking summation with respect to $j$ with the weight $2^{js}$, and noticing that 
\begin{align*}
\overline{\sum} \Big( \int_{\vert \xi\vert \le 1} \Vert \nu^{-(\ell+J_+)/2}\phi(2^{-j}\xi)\hat{f}_0(\xi)\Vert_{L^2_v}^{2p} d\xi \Big)^{1/2p}
\end{align*}
is bounded by $C \Vert \nu^{-(\ell+J_+)/2} f_0\Vert_{L^2_vL^q_x}$ by the same reason as in the proof of Lemma \ref{lem:time-decay of semigroup}, the desired estimate \eqref{ad.lem.td} then follows. This completes the proof of Lemma \ref{lem: Time-decay of velocity-weighted square integrable space}.  
\end{proof}

By Lemma \ref{lem: Time-decay of velocity-weighted square integrable space} together with the $L^2_v$-$L^\infty_v$ interpolation, one can further  derive the time-decay of solutions in the space $\tilde{L}^\infty_v (B^s_{2,1})$ with a suitable velocity weight.

\begin{lem}\label{lem: Time-decay of velocity-weighted infty space}
Assume $-d<\gamma<0$, $0<\beta <\vert \gamma\vert+2$, and $\beta'\ge 0$. Then the solution $f(t,x,v)$ to the Cauchy problem \eqref{eqn: Linearized Equation} with initial data $f_0(x,v)$ satisfies
\begin{align}\label{ad.lem.c0}
\Vert f(t)\Vert_{\tilde{L}^\infty_{\beta'} (B^s_{2,1})}\le C(1+t)^{\beta/\gamma}\left( \Vert f_0 \Vert_{\tilde{L}^\infty_{\beta+\beta'}(B^s_{2,1})}+|\!|\!| f|\!|\!|_{\beta/\vert\gamma\vert, \tilde{L}^2_v (B^s_{2,1})}\right),
\end{align}
for all $t\geq 0$.
\end{lem}

\begin{proof}
We shall follow the proof of \cite[Lemma 4.5]{DHWY}. 
First, due to $L=\nu-K$, we write the linearized equation in the form of 
\begin{align*}
\partial_t f +v\cdot \nabla_x f+\nu f= Kf.
\end{align*}
Define $h(t,x,v)=\langle v\rangle^{\beta'}f(t,x,v)$. Then the equation for $h$ reads
\begin{align*}
\partial_t h + v\cdot \nabla_x h+\nu h= K_{\beta'}h,
\end{align*}
where we have defined 
\begin{align*}
K_{\beta'}(h)(v)=\langle v\rangle^{\beta'}\ K\Big( \frac{h}{\langle \cdot\rangle^{\beta'}}\Big) (v)=\int_{\mathbb{R}^d} k_{\beta'}(v,v')h(v')dv',
\end{align*}
with a new integral kernel $k_{\beta'}(v,v')=k(v,v')\langle v\rangle^{\beta'}/\langle v'\rangle^{\beta'}$.
Therefore, to show the desired estimate \eqref{ad.lem.c0} it suffices to prove
\begin{align}\label{ineq: Weighted Estimate of h}
\Vert h(t)\Vert_{\tilde{L}^\infty_0 (B^s_{2,1})}\le C(1+t)^{\beta/\gamma}\Big( \Vert h_0 \Vert_{\tilde{L}^\infty_{\beta}(B^s_{2,1})}+|\!|\!| f|\!|\!|_{\beta/\vert\gamma\vert, \tilde{L}^2_v (B^s_{2,1})}\Big),
\end{align}
for all $t\geq 0$. Indeed, the mild form of the equation for $h$ is written as 
\begin{align}\label{a.mf}
h(t,x,v)=e^{-\nu(v)t}h_0(x-vt,v)+ \int^t_0 e^{-\nu(v)(t-s)} (K^m_{\beta'}+K^c_{\beta'})h(s, x-(t-s)v,v)ds,
\end{align}
where we have denoted 
$$
(K^m_{\beta'} h)(v)=\int k_{\beta'}^m(v,v_*)\chi_m(\vert v-v_*\vert)h(v_*)dv_*
$$ 
with
\begin{align*}
0\le \chi_m \le 1, \chi_m(t)=1\ \mathrm{for}\ t\le m,\ \chi_m(t)=0\ \mathrm{for}\ t\ge 2m,
\end{align*}
and $K^c_{\beta'}=K_{\beta'}-K^m_{\beta'}$. The small constant $m>0$ will be chosen later.
Applying $\Delta_j$ to \eqref{a.mf} and taking the $L^2_x$-norm, we have
\begin{align}
\Vert \Delta_j h(t,v)\Vert_{L^2_x} &\le e^{-\nu(v)t}\Vert \Delta_j h_0 (v) \Vert_{L^2_x} + \int^t_0 \Vert \Delta_j (K^m_{\beta'} h)(s,v)\Vert_{L^2_x}ds+\int^t_0 \Vert \Delta_j (K^c_{\beta'} h)(s,v)\Vert_{L^2_x}ds\notag\\
&=:L^j_1+L^j_2+L^j_3.\label{a.L123}
\end{align}
To the end, for brevity we put $\tilde{\alpha}=\beta/\vert \gamma\vert>0$ and 
$$
|\!|\!| h|\!|\!|=|\!|\!| h|\!|\!|_{\beta/\vert\gamma\vert, \tilde{L}^\infty_v (B^s_{2,1})}=|\!|\!| h|\!|\!|_{\tilde{\alpha}, \tilde{L}^\infty_v (B^s_{2,1})}.
$$
Notice $0<\tilde{\alpha}< 1-2/\gamma$.


Before starting the estimates on $L^j_k$ $(k=1,2,3)$ in \eqref{a.L123}, we recall some useful facts for $K^m_{\beta'}$ and $K^c_{\beta'}$ in the following lemma, cf.~\cite{DHWY}.

\begin{lem}\label{lem: DHWY}
It holds that
\begin{align}
\vert (K^m_{\beta'} g)(v)\vert &\le C m^{d+\gamma}e^{-\vert v\vert^2/10}\Vert g\Vert_{L^\infty}, \label{ineq: Km}\\
(K^c_{\beta'} g)(v) &=\int_{\mathbb{R}^d} \ell_{\beta'}^c (v,\eta)g(\eta)d\eta,\notag
\end{align}
where the integral kernel $\ell_{\beta'}^c:=\ell^c(v,v')\langle v\rangle^{\beta'}/\langle v'\rangle^{\beta'}$ satisfies
\begin{align}
&\int_{\mathbb{R}^d}\vert \ell_{\beta'}^c(v,\eta)\vert d\eta \le C_\gamma m^{\gamma-1}\frac{\nu(v)}{1+\vert v\vert^2} ,\label{ineq: Kc1}\\
&\int_{\mathbb{R}^d}\vert \ell_{\beta'}^c (v,\eta)\vert e^{-\vert \eta\vert^2/20}d\eta\le C e^{-\vert v\vert^2/100}. \label{ineq: Kc2}
\end{align}
\end{lem}

Now, since it holds that $x^a e^{-x}\le C_a$ on $\{x\ge 0\}$ for each $a\ge 0$, we have
\begin{align*}
\overline{\sum} L^j_1 &\le C \overline{\sum}e^{-\langle v \rangle^\gamma t }\langle v\rangle^{-\beta} \cdot \langle v\rangle^\beta \Vert \Delta_j h_0(v)\Vert_{L^2_x}\\
&\le C\overline{\sum}e^{-\langle v \rangle^\gamma t }(\langle v\rangle^\gamma t)^{\tilde{\alpha}}\  t^{-\tilde{\alpha}} \Vert \Delta_j h_0(v)\Vert_{L^\infty_\beta L^2_x} \le C (1+t)^{-\tilde{\alpha}}\Vert h_0\Vert_{\tilde{L}^\infty_\beta (B^s_{2,1})}.
\end{align*}
By \eqref{ineq: Km} it holds that 
\begin{align*}
\vert (K^m_{\beta'} 1)(v)\vert \le C m^{d+\gamma}e^{-\vert v\vert^2/10}\langle v \rangle^{\beta'} \le C_{\beta'} m^{d+\gamma}e^{-\vert v\vert^2/15}.
\end{align*}
Thus we have
\begin{align*}
\overline{\sum} L^j_2 
&\le C_{\beta'}m^{d+\gamma} e^{-\vert v\vert^2/15}|\!|\!| h|\!|\!| \int^t_0 e^{-\nu(v)(t-s)} (1+s)^{-\tilde{\alpha}} ds \\
&\le C_{\beta'}m^{d+\gamma} e^{-\vert v\vert^2/20}|\!|\!| h|\!|\!| \int^t_0 (1+t-s)^{-\tilde{\alpha}-1} (1+s)^{-\tilde{\alpha}} ds \\
&\le C_{\beta'}m^{d+\gamma} e^{-\vert v\vert^2/20}|\!|\!| h|\!|\!| (1+t)^{-\tilde{\alpha}},
\end{align*}
where we have used the inequality $e^{-\vert v\vert^2/10}e^{-\nu(v)(t-s)}\le C_b e^{-\vert v\vert^2/20}(1+t-s)^{-b}$ for $b\ge 0$. This then completes the estimates on $L^j_1$ and $L^j_2$. Furthermore, by substituting those estimates into $L^j_3$, one has
\begin{align*}
L^j_3 &\le \int^t_0 e^{-\nu(v) (t-s)} \int_{\mathbb{R}^d} \vert \ell_{\beta'}^c (v,v')\vert \Big[ e^{-\nu(v')s}\Vert \Delta_j h_0 (v')\Vert_{L^2_x} \\
&\qquad\qquad\qquad\qquad+ \int^s_0 e^{-\nu(v')(s-\tau)} \Vert \Delta_j (K^m_{\beta'} h)(\tau, v')\Vert_{L^2_x} d\tau \Big] dv'ds\\
&\quad+ \int^t_0 e^{-\nu(v)(t-s)} \int^s_0 e^{-\nu(v')(s-\tau)}  \int_{\mathbb{R}^d\times \mathbb{R}^d} \vert \ell_{\beta'}^c (v,v') \ell_{\beta'}^c (v', v'')\vert  \Vert \Delta_j h(\tau, v'')\Vert_{L^2_x}dv''dv'd\tau ds\\
&=:L^j_{31}+L^j_{32}+L^j_{33}.
\end{align*}
Here $L^j_{31}$ and $L^j_{32}$ can be similarly estimated as $L^j_1$ and $L^j_2$, respectively. In fact, it follows from \eqref{ineq: Kc1} that 
\begin{align*}
\overline{\sum} L^j_{31} &\le \Vert h_0\Vert_{\tilde{L}^\infty_\beta (B^s_{2,1})}\int^t_0 \int_{\mathbb{R}^d} e^{-\nu(v)(t-s)-\nu(v')s} \langle v' \rangle^{-\beta} \vert \ell_{\beta'}^c (v,v')\vert dv'ds\\
&\le C\Vert h_0\Vert_{\tilde{L}^\infty_\beta (B^s_{2,1})}\int^t_0 \int_{\mathbb{R}^d} e^{-\nu(v)(t-s)}(1+s)^{-\tilde{\alpha}} \vert \ell_{\beta'}^c (v,v')\vert dv'ds\\
&\le C_\gamma m^{\gamma -1}\Vert h_0\Vert_{\tilde{L}^\infty_\beta (B^s_{2,1})}\int^t_0 e^{-\nu(v)(t-s)}(1+s)^{-\tilde{\alpha}}\frac{\nu(v)}{1+\vert v\vert^2}ds\\
&\le C_\gamma m^{\gamma -1}\Vert h_0\Vert_{\tilde{L}^\infty_\beta (B^s_{2,1})}\int^t_0 (1+t-s)^{-1+2/\gamma}(1+s)^{-\tilde{\alpha}}ds\\
&\le C_\gamma m^{\gamma -1}\Vert h_0\Vert_{\tilde{L}^\infty_\beta (B^s_{2,1})}(1+t)^{-\tilde{\alpha}}.
\end{align*}
And, by \eqref{ineq: Km} and \eqref{ineq: Kc2}, one has
\begin{align*}
&\overline{\sum} L^j_{32} \le Cm^{d+\gamma}\overline{\sum}\int^t_0\int^s_0 \int_{\mathbb{R}^d} e^{-\nu(v)(t-s)-\nu(v')s} e^{-\vert v'\vert^2/10} \Vert \Delta_j h(\tau)\Vert_{L^\infty_v L^2_x} \vert \ell_{\beta'}^c (v,v')\vert dv'd\tau ds\\
&\le Cm^{d+\gamma} |\!|\!| h|\!|\!| \int^t_0\int^s_0 \int_{\mathbb{R}^d} e^{-\nu(v)(t-s)}(1+s-\tau)^{-1-\tilde{\alpha}}  e^{-\vert v'\vert^2/20}(1+\tau)^{-\tilde{\alpha}}\vert \ell_{\beta'}^c (v,v')\vert dv'd\tau ds\\
&\le C_\gamma m^{d+\gamma} |\!|\!| h|\!|\!| \int^t_0\int^s_0 e^{-\nu(v)(t-s)}e^{-\vert v\vert^2/100}(1+s-\tau)^{-1-\tilde{\alpha}}(1+\tau)^{-\tilde{\alpha}} d\tau ds\\
&\le C_\gamma m^{d+\gamma} |\!|\!| h|\!|\!| \int^t_0\int^s_0 (1+t-s)^{-1-\tilde{\alpha}}(1+s-\tau)^{-1-\tilde{\alpha}} (1+\tau)^{-\tilde{\alpha}} d\tau ds\\
&\le C_\gamma m^{d+\gamma} |\!|\!| h|\!|\!| (1+t)^{-\tilde{\alpha}}.
\end{align*}
To estimate $L^j_{33}$, we divide it by three cases.
First, if $\vert v \vert \ge N$, it holds that 
\begin{align*}
&\overline{\sum} \int^t_0  \int^s_0 \int_{\mathbb{R}^d\times \mathbb{R}^d} e^{-\nu(v)(t-s)}e^{-\nu(v')(s-\tau)} \vert \ell_{\beta'}^c (v,v') \ell_{\beta'}^c (v', v'')\vert  \Vert \Delta_j h(\tau)\Vert_{L^\infty_v L^2_x}dv''dv'd\tau ds\\
&\le C_\gamma m^{\gamma-1} |\!|\!| h|\!|\!| \int^t_0  \int^s_0 \int_{\mathbb{R}^d} e^{-\nu(v)(t-s)}e^{-\nu(v')(s-\tau)}(1+\tau)^{-\tilde{\alpha}} \frac{\nu(v')}{1+\vert v'\vert^2}  \vert \ell_{\beta'}^c (v,v')\vert dv'd\tau ds\\
&\le C_\gamma m^{\gamma-1} |\!|\!| h|\!|\!| \int^t_0 \int^s_0 e^{-\nu(v)(t-s)}(1+s-\tau)^{-1+2/\gamma}(1+\tau)^{-\tilde{\alpha}} \frac{\nu(v)}{1+\vert v\vert^2} d\tau ds\\
&\le C_\gamma m^{\gamma-1}N^{-\delta \vert \gamma\vert}|\!|\!| h|\!|\!| \int^t_0 \int^s_0 (1+t-s)^{-1+2/\gamma+\delta}(1+s-\tau)^{-1+2/\gamma}(1+\tau)^{-\tilde{\alpha}} d\tau ds\\
&\le C_\gamma m^{\gamma-1}N^{-\delta \vert \gamma\vert} |\!|\!| h|\!|\!| (1+t)^{-\tilde{\alpha}}.
\end{align*}
Here, $\delta>0$ is a suitably small constant such that  both $0< \tilde{\alpha}\le 1-2/\gamma-\delta$ and $1-2/\gamma -\delta>1$ hold true. Notice that such a constant $\delta>0$ exists by the assumptions of Lemma \ref{lem: Time-decay of velocity-weighted infty space}.

The second case is to consider either $\{\vert v\vert\le N$, $\vert v'\vert\ge 2N\}$ or $\{\vert v'\vert\le 2N$, $\vert v''\vert\ge 3N\}$. For simplicity we only consider the former one since the proof for the latter one is almost the same. Recall that
\begin{align*}
\vert \ell_{\beta'}^c(v,v')\vert \le C e^{-N^2/20} \vert \ell_{\beta'}^c(v,v')\vert e^{\vert v-v'\vert^2/20},\\
\int_{\mathbb{R}^d}\vert \ell_{\beta'}^c(v,v')\vert e^{\vert v-v'\vert^2/20}dv' \le C_m \frac{\nu(v)}{1+\vert v\vert^2},
\end{align*}
where the second estimate has been shown in  \cite{DHWY}. Therefore, $\sum 2^{js}L^j_{33}$ is bounded by 
\begin{align*}
&C_\gamma m^{\gamma-1} \overline{\sum} \int^t_0  \int^s_0 \int_{\mathbb{R}^d} e^{-\nu(v)(t-s)}e^{-\nu(v')(s-\tau)} \vert \ell_{\beta'}^c (v,v')\vert \frac{\nu(v')}{1+\vert v'\vert^2}\Vert \Delta_j h(\tau)\Vert_{L^\infty_v L^2_x} dv'd\tau ds\\
&\le C_\gamma m^{\gamma-1} |\!|\!| h|\!|\!| \int^t_0  \int^s_0 e^{-\nu(v)(t-s)} e^{-N^2/20} \frac{\nu(v)}{1+\vert v\vert^2}(1+s-\tau)^{-1+2/\gamma}(1+\tau)^{-\tilde{\alpha}}d\tau ds\\
&\le C_\gamma m^{\gamma-1}e^{-N^2/20} |\!|\!| h|\!|\!|(1+t)^{-\tilde{\alpha}}.
\end{align*}

Third, if $\vert v\vert\le N$, $\vert v'\vert\le 2N$, and $\vert v''\vert\le 3N$, then we take a small constant $\lambda >0$ to be chosen later. We divide the $\tau$-integration into two parts $\int^s_0=\int^s_{s-\lambda}+\int^{s-\lambda}_0$. For the first integral $\int^s_{s-\lambda}$, we notice
\begin{align*}
\int^s_{s-\lambda}e^{-\nu(v')(s-\tau)}(1+\tau)^{-\tilde{\alpha}}d\tau \le C\lambda (1+s)^{-\tilde{\alpha}},
\end{align*}
where $C$ is independent of $\lambda$. 
 Therefore, $\sum 2^{js}L^j_{33}$ is correspondingly dominated by
\begin{align*}
&C\lambda |\!|\!| h|\!|\!| \int^t_0\int_{\vert v'\vert\le 2N, \vert v''\vert \le 3N} (1+s)^{-\tilde{\alpha}} e^{-\nu(v)(t-s)}  \vert \ell_{\beta'}^c (v,v')\ell_{\beta'}^c (v', v'')\vert dv'' dv'ds\\
&\le C_\gamma m^{2-2\gamma}\lambda |\!|\!| h|\!|\!|\int^t_0  e^{-\nu(v)(t-s)} \Big(\frac{\nu(v)}{1+\vert v\vert^2}\Big)^2 (1+s)^{-\tilde{\alpha}}ds\\
&\le C_\gamma m^{2-2\gamma}\lambda  |\!|\!| h|\!|\!|\int^t_0 (1+t-s)^{-2(1-2/\gamma)} (1+s)^{-\tilde{\alpha}}ds\\
&\le C_\gamma m^{2-2\gamma}\lambda  |\!|\!| h|\!|\!|(1+t)^{\beta/\gamma},
\end{align*}
where the estimate \eqref{ineq: Kc1} in Lemma \ref{lem: DHWY} has been used twice in the first inequality. For the second integral $\int^{s-\lambda}_0$, we notice that one can take $\tilde{\ell}_{\beta', N} \in C^\infty_0(\mathbb{R}^d\times \mathbb{R}^d)$ satisfying 
\begin{align*}
\sup_{\vert p\vert\le 3N} \int_{\vert v'\vert\le 3N} \vert \ell_{\beta'}^c(p,v')-\tilde{\ell}_{\beta', N} (p,v')\vert dv' \le C_m N^{\gamma-1}.
\end{align*}
With this approximation function, we decompose the product $\ell_{\beta'}^c (v,v')\ell_{\beta'}^c(v',v'')$ into
\begin{align*}
\ell_{\beta'}^c (v,v')\ell_{\beta'}^c(v',v'') &= (\ell_{\beta'}^c (v,v')-\tilde{\ell}_{\beta',N}(v,v'))\ell_{\beta'}^c (v',v'')\\
&+(\ell_{\beta'}^c (v',v'')-\tilde{\ell}_{\beta', N}(v',v''))\tilde{\ell}_{\beta', N} (v,v')\\
&+ \tilde{\ell}_{\beta', N}(v,v')\tilde{\ell}_{\beta', N}(v',v'').
\end{align*}
The integral with the kernel $(\ell_{\beta'} (v,v')-\tilde{\ell}_{\beta',N}(v,v'))\ell_{\beta'} (v',v'')$ is bounded by
\begin{align*}
& |\!|\!| h|\!|\!|\int^t_0 \int^{s-\lambda}_0 \int_{\vert v'\vert \le 2N, \vert v''\vert\le 3N}(1+\tau)^{-\tilde{\alpha}} e^{-\nu(v) (t-s)}e^{-\nu(v')(s-\tau)}\\ 
&\qquad\qquad\qquad\qquad\times \vert \ell_{\beta'}^c (v,v')-\tilde{\ell}_{\beta',N}(v,v')\vert \vert\ell_{\beta'} (v',v'')\vert dv''dv'd\tau ds\\
&\le  C_\gamma m^{\gamma-1}|\!|\!| h|\!|\!| \int^t_0 \int^{s-\lambda}_0 \int_{\vert v'\vert \le 2N} (1+\tau)^{-\tilde{\alpha}} e^{-\nu(v) (t-s)}e^{-\nu(v')(s-\tau)} \\
&\qquad\qquad\qquad\qquad\times \frac{\nu(v')}{1+\vert v'\vert^2} \vert\ell_{\beta'}^c (v,v')-\tilde{\ell}_{\beta', N}(v,v')\vert dv'd\tau ds\\
&\le C_{\gamma} m^{\gamma-1}N^{\gamma-1}|\!|\!| h|\!|\!| \int^t_0 \int^{s-\lambda}_0 e^{-N^\gamma (t-s)} (1+s)^{-\tilde{\alpha}} (1+s-\tau)^{-1+2/\gamma} d\tau ds\\
&\le C_{\gamma} m^{\gamma-1}N^{-1}|\!|\!| h|\!|\!| (1+t)^{-\tilde{\alpha}},
\end{align*}
where we have used the fact that $\nu(v) \geq  cN^\gamma$ if $\vert v\vert\le N$. The estimate on the second term is similar and simpler, because $\tilde{\ell}_{\beta', N}(v,v')$ is not singular.
Also, in terms of boundedness of $\tilde{\ell}_{\beta', N}$ and the Cauchy-Schwarz inequality, we have
\begin{align*}
&C_N\overline{\sum} \int^t_0 \int^{s-\lambda}_0 \int_{\vert v'\vert \le 2N, \vert v''\vert\le 3N} e^{-\nu(v) (t-s)}e^{-\nu(v')(s-\tau)} \\
&\qquad\qquad\qquad\qquad\qquad \times \vert \tilde{\ell}_{\beta', N}(v,v')\tilde{\ell}_{\beta', N}(v',v'')\vert\Vert \Delta_j h(\tau, v'')\Vert_{L^2_x}dv'' dv' d\tau ds\\
&\le C_N \overline{\sum} \int^t_0 \int^{s-\lambda}_0 e^{-cN^\gamma (t-s)} e^{-cN^\gamma (s-\tau)}\Vert \Delta_j f(\tau) \Vert_{L^2_{x,v}}d\tau ds\\
&\le C_N |\!|\!| f|\!|\!|_{\tilde{\alpha}, \tilde{L}^2_v (B^s_{2,1})} \int^t_0 \int^{s-\lambda}_0 e^{-cN^\gamma (t-s)} e^{-cN^\gamma (s-\tau)}(1+\tau)^{-\tilde{\alpha}} d\tau ds\\
&\le C_{\gamma,N} |\!|\!| f|\!|\!|_{\tilde{\alpha}, \tilde{L}^2_v (B^s_{2,1})} (1+t)^{-\tilde{\alpha}}.
\end{align*}
Here, once again we have used  the fact that $\nu(v)$, $\nu(v')\geq  cN^\gamma$ if $\vert v\vert\le N$ and $\vert v'\vert \le 2N$. Also, boundedness of the integral domain has reduced the $L^2_v$-estimate of $h(v)=\langle v\rangle^{\beta'}f(v)$ to that of $f(v)$.

Finally, summing up all the above estimates, we obtain
\begin{align*}
\Vert \Delta_j h(t,v)&\Vert_{L^2_x} \le C (1+t)^{-\tilde{\alpha}} \big(\Vert h_0\Vert_{\tilde{L}^\infty_0 (B^s_{2,1})}+ |\!|\!| f|\!|\!|_{\tilde{\alpha}, \tilde{L}^2_v (B^s_{2,1})} \big)\\ 
&+C' (1+t)^{-\tilde{\alpha}}\big( m^{2-2\gamma}\lambda+m^{\gamma-1}(N^{-1}+e^{-N^2/20}+N^{-\delta\vert \gamma\vert})+m^{d+\gamma}\big)  |\!|\!| h|\!|\!|,
\end{align*}
where $C=C(\gamma, m, N)>0$ and $C'>0$ is independent of $(\gamma, m, N)$.
Now, by taking first $m>0$ small, next $\lambda>0$ sufficiently small, and then $N>0$ sufficiently large, we then derive the desired estimate \eqref{ineq: Weighted Estimate of h}. This completes the proof of Lemma \ref{lem: Time-decay of velocity-weighted infty space}. 
\end{proof}

Combining Lemma \ref{lem: Time-decay of velocity-weighted square integrable space} and Lemma \ref{lem: Time-decay of velocity-weighted infty space} immediately yields the following

\begin{cor}\label{cor: Time-decay of the linearized solution for soft}
Let $q\in [1,2]$, $\beta \ge 0$, and  $\alpha = d/2(1/q-1/2)$. 
 Then the solution $f(t,x,v)$ to the linearized Cauchy problem \eqref{eqn: Linearized Equation} with initial data $f_0(x,v)$ satisfies
\begin{align*}
\Vert f\Vert_{\tilde{L}^\infty_{\beta} (B^s_{2,1})} \le C(1+t)^{-\alpha} \Big( \Vert f_0\Vert_{\tilde{L}^\infty_{\alpha\vert\gamma\vert+\beta} (B^s_{2,1})} +\Vert \nu^{-\alpha_+} f_0 \Vert_{\tilde{L}^2_v (B^s_{2,1})} + \Vert \nu^{-\alpha_+} f_0 \Vert_{L^2_vL^q_x}\Big).
\end{align*}
\end{cor}

We shall apply the preceding statements for the linear problem to the nonlinear one.

\begin{thm}\label{thm: Time-Decay of the nonlinear solution for soft potential}
Assume $s\ge d/2$, $q =1$, $\beta\ge 0$ and $\beta > (1-\alpha/2)\gamma + d/2=(1-d/4)\gamma+d/2$. Then the solution $f(t,x,v)$ to the mild form of the Cauchy problem on the nonlinear Boltzmann equation
\begin{align*}
f(t)=e^{tB}f_0 +\int^t_0 e^{(t-s)B} \Gamma (f,f)(s) ds
\end{align*}
enjoys the following estimate:
\begin{align}
\Vert f(t)\Vert_{\tilde{L}^\infty_\beta (B^s_{2,1})} &\le C(1+t)^{-d/4} \Vert f_0\Vert_{\tilde{L}^\infty_{(\beta+d\vert \gamma \vert /4)}(B^s_{2,1})\cap \tilde{L}^2_{(d\vert\gamma\vert/4)_+}(B^s_{2,1})\cap L^2_{((d\vert\gamma\vert//4)_+} L^1_x}\notag\\
&+C(1+t)^{-d/4} |\!|\!| f|\!|\!|_{d/4, \tilde{L}^\infty_\beta (B^s_{2,1})}^2.\label{ad.thm.e1}
\end{align}
\end{thm}

\begin{proof}
Owing to Corollary \ref{cor: Time-decay of the linearized solution for soft}, $\Vert e^{tB}f_0 \Vert_{\tilde{L}^\infty_\beta (B^s_{2,1})}$ can be bounded by the first term on the right-hand side of \eqref{ad.thm.e1}. Thus it suffices to consider the estimate of
\begin{align}\label{term: Time-integral of nonlinear term}
\int^t_0 (1+t-s)^{-d/4} \Vert \Gamma(f,f)\Vert_{\tilde{L}^\infty_\beta (B^s_{2,1})}ds.
\end{align}

First, we claim that for $s\ge d/2$ and $(\beta_1, \beta_2) \in \mathbb{R}^2$ with $\gamma +\beta_1 \le \beta_2$, it holds that 
\begin{align}\label{ineq: Weighted estimate of Gamma 1}
\Vert \Gamma(f,g)\Vert_{\tilde{L}^\infty_{\beta_1} (B^s_{2,1})} \le C\Vert f\Vert_{\tilde{L}^\infty_{\beta_2}(B^s_{2,1})}\Vert g\Vert_{\tilde{L}^\infty_{\beta_2}(B^s_{2,1})}.
\end{align}
Indeed, the proof is similar to that of Lemma \ref{lem:bilinear estimate of Gamma}, so we only show the estimate of the term involving $\Gamma^1_{loss}(f,g)$ for brevity. Then, it holds that 
\begin{align*}
&\Vert \Gamma^1_{loss}(f,g)\Vert_{\tilde{L}^\infty_{\beta_1} (B^s_{2,1})}\\
&\le C \overline{\sum} \sum_{\vert i-j\vert\le 4} \sup_{v} \langle v\rangle^{\beta_1} \Big ( \int_{\mathbb{R}^d} \Big( \int_{\mathbb{R}^d} \vert v-v_*\vert^\gamma M_*^{1/2} \Delta_j (\Delta_i f_* S_{i-1}g) dv_*\Big)^2 dx\Big)^{1/2}\\
&\le C \overline{\sum} \sum_{\vert i-j\vert\le 4} \sup_{v} \langle v\rangle^{\beta_1}\int_{\mathbb{R}^d} \vert v-v_*\vert^\gamma M_*^{1/2} \Vert \Delta_i f_* S_{i-1}g\Vert_{L^2_x} dv_*\\
&\le C \Vert f\Vert_{\tilde{L}^\infty_{\beta_2}(B^s_{2,1})} \Vert g\Vert_{L^\infty_{\beta_2} L^\infty_x} \sup_v \langle v\rangle^{\beta_1} \int_{\mathbb{R}^d} \vert v-v_*\vert^\gamma M_*^{1/2} \langle v_*\rangle^{-\beta_2}\langle v\rangle^{-\beta_2} dv_*\\
&\le C \Vert f\Vert_{\tilde{L}^\infty_{\beta_2}(B^s_{2,1})}\Vert g\Vert_{\tilde{L}^\infty_{\beta_2}(B^s_{2,1})}\sup_v \langle v\rangle^{\beta_1-\beta_2+\gamma}.
\end{align*}
Here, the supremum in the last line is finite thanks to $\gamma +\beta_1 \le \beta_2$.

Second, we also claim that for $s\ge d/2$, $\beta_1\in\mathbb{R}$, and $\beta_2 \ge 0$  with $\gamma + d/2 +\beta_1 <\beta_2$, it holds that 
\begin{align}\label{ineq: Weighted estimate of Gamma 2}
\Vert \Gamma(f,g) \Vert_{\tilde{L}^2_{\beta_1}(B^s_{2,1})} \le C \Vert f\Vert_{\tilde{L}^\infty_{\beta_2}(B^s_{2,1})}\Vert g\Vert_{\tilde{L}^\infty_{\beta_2}(B^s_{2,1})}.
\end{align}
This is an improved version of Lemma \ref{lem:bilinear estimate of Gamma}, and the proof is almost the same.
It only suffices to verify the boundedness of
\begin{align*}
\int_{\mathbb{R}^3} \langle v\rangle^{2(\beta_1 -\beta_2 +\gamma)}dv,
\end{align*}
with the help of the suitable choice of $\beta_1$ and $\beta_2$ such that $\gamma + d/2 +\beta_1 <\beta_2$.
We remark that the non-negativity of $\beta_2$ is required to apply \eqref{ineq: Pre-Post Bessel Potential} to the gain terms.

Third, as for showing 
 \eqref{ineq: Weighted estimate of Gamma 2}, one has
\begin{align}\label{ineq: Weighted estimate of Gamma 3}
\Vert \Gamma (F,G)\Vert_{L^2_{\beta_1}} \le \Vert \Gamma_{gain} (F,G)\Vert_{L^2_{\beta_1}} +\Vert \Gamma_{loss} (F,G)\Vert_{L^2_{\beta_1}} \le C \Vert F\Vert_{L^\infty_{\beta_2}}\Vert G\Vert_{L^\infty_{\beta_2}},
\end{align}
for $\gamma + d/2 +\beta_1 <\beta_2$ with $\beta_1\in\mathbb{R}$ and $\beta_2 \ge 0$.

Now, applying Corollary \ref{cor: Time-decay of the linearized solution for soft}, \eqref{term: Time-integral of nonlinear term} is bounded by
\begin{align}\label{term: Estimate of time integral of Gamma}
C\int^t_0 (1+t-s)^{-d/4} \Big( \Vert \Gamma (f,f)(s)\Vert_{\tilde{L}^\infty_{\beta}(B^s_{2,1})}&+\Vert \nu^{-(d/4)_+} \Gamma (f,f)(s)\Vert_{\tilde{L}^2_{v}(B^s_{2,1})}\notag\\
& + \Vert \nu^{-(d/4)_+} \Gamma (f,f)(s)\Vert_{L^2_vL^1_x}\Big) ds.
\end{align}
Each norm in the above integral can be estimated in the following way.
Setting $\beta_1=\beta_2=\beta$ in \eqref{ineq: Weighted estimate of Gamma 1}, which is possible due to $\gamma<0$, we have
\begin{align*}
\Vert \Gamma (f,f)(s)\Vert_{\tilde{L}^\infty_{\beta}(B^s_{2,1})}\le C \Vert f(s)\Vert_{\tilde{L}^\infty_\beta(B^s_{2,1})}^2.
\end{align*}
Setting $\beta_1=d\vert\gamma\vert/4$ and $\beta_2=\beta$ in \eqref{ineq: Weighted estimate of Gamma 2} gives
\begin{align*}
\Vert \Gamma(f,f)(s) \Vert_{\tilde{L}^2_{(d\vert\gamma\vert/4)_+}(B^s_{2,1})}\le C\Vert f(s)\Vert_{\tilde{L}^\infty_{\beta}(B^s_{2,1})}^2.
\end{align*}
Furthermore, setting $F=G=\Vert f\Vert_{L^2_x}$, $\beta_1=d\vert\gamma\vert/4$ and $\beta_2=\beta$ in \eqref{ineq: Weighted estimate of Gamma 3} yields
\begin{align*}
\Vert \nu^{-(d/4)_+} \Gamma (f,f)(s)\Vert_{L^2_vL^1_x}\le C \Vert f(s)\Vert_{L^\infty_{\beta}L^2_x}^2 \le C \Vert f(s)\Vert_{\tilde{L}^\infty_{\beta}(B^s_{2,1})}^2.
\end{align*}
Therefore, plugging those inequalities back into \eqref{term: Estimate of time integral of Gamma}, we obtain
\begin{align*}
C\int^t_0 (1+t-s)^{-d/4}\Vert f(s)\Vert_{L^\infty_{\beta}L^2_x}^2ds  &\le C|\!|\!| f|\!|\!|_{d/4, \tilde{L}^\infty_\beta (B^s_{2,1})}^2 \int^t_0 (1+t-s)^{-d/4} (1+s)^{-d/2} ds\\
&\le C (1+t)^{-d/4}|\!|\!| f|\!|\!|_{d/4, \tilde{L}^\infty_\beta (B^s_{2,1})}^2,
\end{align*}
for $d\ge 3$. This then proves the desired estimate \eqref{ad.thm.e1} and completes the proof of Theorem \ref{thm: Time-Decay of the nonlinear solution for soft potential}.
\end{proof}

The theorem above provides the global a priori estimates stated in the following

\begin{cor}\label{cor: globa a priori estimate}
Assume $-d<\gamma<0$, $q=1$, $s\ge d/2$, $\beta \ge 0$ and $\beta > (1-\alpha/2)\gamma + d/2=(1-d/4)\gamma+d/2$. Then there exist $\varepsilon>0$ and $C>0$ such that if
\begin{align*}
\Vert f_0\Vert_{\tilde{L}^\infty_{(\beta+d\vert\gamma\vert/4)}(B^s_{2,1})\cap \tilde{L}^2_{(d\vert\gamma\vert/4)_+}(B^s_{2,1})\cap L^2_{((d\vert\gamma\vert/4)_+} L^1_x} \le  \varepsilon ,
\end{align*}
then the solution $f(t,x,v)$ to the Boltzmann equation with initial datum $f_0(x,v)$ satisfies
\begin{align*}
|\!|\!| f|\!|\!|_{3/4, \tilde{L}^\infty_\beta (B^s_{2,1})} \le C \Vert f_0\Vert_{\tilde{L}^\infty_{(\beta+d\vert\gamma\vert/4)}(B^s_{2,1})\cap \tilde{L}^2_{(d\vert\gamma\vert/4)_+}(B^s_{2,1})\cap L^2_{((d\vert\gamma\vert/4)_+} L^1_x}.
\end{align*}
\end{cor}
Together with the inclusion $L^\infty_{\beta_1} \hookrightarrow L^2_{\beta_2}$ for $\beta_1>\beta_2+d/2$ and the local-in-time existence whose proof will be postponed to the next section, Corollary \ref{cor: globa a priori estimate}  yields Theorem \ref{thm: Solution for the soft potential case} with the help of the standard continuity argument. \qed

\section{Appendix}

Regarding Theorem \ref{thm: Solution for the soft potential case} for the soft potential case, in order to establish the local-in-time existence of solutions, we will follow the strategy of \cite{DLX}, and give the full details of the proof for completeness. The approximation scheme is given by
\begin{equation*}
\left\{\begin{aligned}
&\dis (\partial_t + v\cdot \nabla_x)F^{n+1} + F^{n+1} \int_{\mathbb{R}^d \times \mathbb{S}^{d-1}} \vert v-v_*\vert^\gamma b_0(\theta) F^n_* dv_* d\omega\\
&\dis \qquad\qquad\qquad\qquad\qquad=\int_{\mathbb{R}^d \times \mathbb{S}^{d-1}} \vert v-v_*\vert^\gamma  b_0(\theta) F'^n_* F'^n  dv_* d\omega,\\
&\dis F^{n+1}|_{t=0}=F_0,
\end{aligned}\right.
\end{equation*}
with $n=0,1,2,\cdots$, where we have set $F^0\equiv M$. 
Plugging $F^n(t,x,v)=M+M^{1/2}f^n(t,x,v)$,
we have the iterative equations:
\begin{equation}\label{eqn: The Approximation Scheme}
\left\{\begin{aligned}
&\dis (\partial_t + v\cdot \nabla_x +\nu)f^{n+1} -Kf^n =\Gamma_{gain}(f^n,f^n)-\Gamma_{loss}(f^n, f^{n+1}),\\
&\dis f^{n+1}|_{t=0}=f_0,
\end{aligned}\right.
\end{equation}
with $n=0,1,2,\cdots$, where $f^0\equiv 0$.

\begin{lem}\label{lem: Approximation functions}
The solution sequence $\{f^n(t,x,v)\}_{n=1}^\infty$ is well-defined. Precisely, let $-d<\gamma <0$, $s\ge d/2$, $\alpha \in \mathbb{R}$, and $\beta \in \mathbb{R}$. Then, there are constants $M_0>0$ and $T^*=T^*(M_0)>0$ such that if initial data $f_0$ satisfies
$\Vert f_0\Vert_{\tilde{L}^\infty_\beta (B^s_{2,1})} \le M_0$
then for any $n$ and $T \in [0,T^*)$, it holds that 
\begin{align}\label{ad.ap.l1}
\mathcal{E}_T(f^n)+\mathcal{D}_T(f^n)\le 2M_0,
\end{align}
where we have denoted
\begin{align*}
&\mathcal{E}_T(f^n)=\sup_{0\le t\le T} (1+t)^\alpha \Vert f^n(t)\Vert_{\tilde{L}^\infty_\beta (B^s_{2,1})},\notag\\
&\mathcal{D}_T(f^n)=\sup_{0\le t\le T} (1+t)^\alpha \overline{\sum} \Big( \int^t_0 \Vert \Delta_j f^{n}(s)\Vert_{L^\infty_{\beta+\gamma/2}L^2_x}^2ds\Big)^{1/2}.
\end{align*}
\end{lem}

\begin{proof}
We shall prove \eqref{ad.ap.l1} by induction in $n$ for a suitable choice of $M_0>0$ to be determined in the end of the proof. Obviously it is true for $n=0$ since $f^0\equiv 0$ by the definition.  
We assume for the fixed $n\geq 0$ that it holds that 
\begin{align}\label{ad.ap.p1}
\mathcal{F}_T(f^n):=\mathcal{E}_T(f^n)+\mathcal{D}_T(f^n) \le 2M_0,
\end{align}
for any $0\le T <T_*$, and shall prove that the above inequality is still valid for $n+1$. We take $T\in [0,T^\ast)$, and write $\mathcal{E}_T(f^n)=\mathcal{E}_T^n$ and $\mathcal{D}_T(f^n)=\mathcal{D}_T^n$ for brevity. 
By applying $\Delta_j$ to \eqref{eqn: The Approximation Scheme}, multiplying the resulting equation with $2^{2sj}\langle v\rangle^{2\beta} \Delta_j f^{n+1}$, and then integrating both sides with respect to $x$, we have
\begin{align}
&\qquad\frac{d}{dt} 2^{2js} \langle v\rangle^{2\beta} \Vert \Delta_j f^{n+1} (t,v)\Vert_{L^2_x}^2 + 2^{2js+1}\langle v\rangle^{2\beta} \nu(v)\Vert \Delta_j f^{n+1}(t,v)\Vert_{L^2_x}^2\notag\\
&= 2^{2js+1}\langle v\rangle^{2\beta} \Big( K\Delta_j f^n + \Delta_j (\Gamma_{gain}(f^n, f^n)- \Gamma_{loss}(f^n, f^{n+1})), \Delta_j f^{n+1}\Big)_x (t,v),\label{eqn: Resultant 1}
\end{align}
where $(\cdot,\cdot)_x$ denotes the inner product of the Hilbert space $L^2_x$. 
By further integrating \eqref{eqn: Resultant 1} over $[0,t]$ for $0<t<T$, taking supremum with respect to $v$, taking the square root, and then taking summation with respect to $j$, it follows that 
\begin{align}\label{ineq: Resultant 2}
&\overline{\sum} \Vert \Delta_j f^{n+1} (t)\Vert_{L^\infty_\beta L^2_x} + {\overline{\sum} \Big( \int^t_0 \Vert \Delta_j f^{n+1} (s)\Vert_{L^\infty_{\beta+\gamma/2}L^2_x}^2ds\Big)^{1/2}}\notag\\
&\le C\overline{\sum} \Vert \Delta_j f_0\Vert_{L^\infty_\beta L^2_x} +\overline{\sum} \Big( \int^t_0 \sup_v \langle v\rangle^{2\beta} \vert (K\Delta_j f^n, \Delta_j f^{n+1})_x (s)\vert ds\Big)^{1/2}\notag\\
&\quad +\overline{\sum} \Big( \int^t_0 \sup_v \langle v\rangle^{2\beta} \vert (\Delta_j \Gamma_{loss}(f^n, f^n), \Delta_j f^{n+1})_x (s)\vert ds\Big)^{1/2}\notag\\
&\quad+\overline{\sum} \Big( \int^t_0 \sup_v \langle v\rangle^{2\beta} \vert (\Delta_j \Gamma_{gain}(f^n, f^n), \Delta_j f^{n+1})_x (s)\vert ds\Big)^{1/2}.
\end{align}
We first consider the estimate on the second term on the right-hand side of \eqref{ineq: Resultant 2}. Note that 
\begin{align*}
\sup_v\langle v\rangle^{2\beta} \vert (K\Delta_j f^n, \Delta_j f^{n+1} )_x \vert \le C \Vert \Delta_j f^n\Vert_{L^\infty_\beta L^2_x}\Vert \Delta_j f^{n+1}\Vert_{L^\infty_\beta L^2_x},
\end{align*}
because $K\in \mathscr{B}(L^\infty_\beta, L^\infty_{\beta+1})$ holds true. Therefore we have
\begin{align}\label{ineq: Trilinear Estimate of K}
&\overline{\sum} \Big( \int^t_0 \sup_v \langle v\rangle^{2\beta} \vert (K\Delta_j f^n, \Delta_j f^{n+1})_x (s)\vert ds\Big)^{1/2}\notag\\
&\le C \overline{\sum} \Big( \int^t_0 \Vert \Delta_j f^n(s)\Vert_{L^\infty_\beta L^2_x}\Vert \Delta_j f^{n+1}(s)\Vert_{L^\infty_\beta L^2_x} ds\Big)^{1/2}\notag\\
&\le C \overline{\sum}\sqrt{\mathcal{E}^n_T} \Big(\int^t_0 2M_0 (1+s)^{-\alpha} 2^{-js} c_j \Vert \Delta_j f^{n+1}(s)\Vert_{L^\infty_\beta L^2_x} ds\Big)^{1/2}\notag\\
&\le C\sqrt{\mathcal{E}^n_t} \sum_{j\ge -1}2^{js/2}c_j^{1/2} \Big(\int^t_0  (1+s)^{-\alpha} \Vert \Delta_j f^{n+1}(s)\Vert_{L^\infty_\beta L^2_x} ds\Big)^{1/2}\notag\\
&\le C\sqrt{\mathcal{E}^n_t}\Big( \sum_{j\ge -1}c_j\Big)^{1/2} \Big(\overline{\sum} \int^t_0 (1+s)^{-\alpha}\Vert \Delta_j f^{n+1}(s)\Vert_{L^\infty_\beta L^2_x} ds\Big)^{1/2}\notag\\
&\le C\sqrt{\mathcal{E}^n_t}\sqrt{\mathcal{E}^{n+1}_t}\Big( \int^t_0 (1+s)^{-2\alpha} ds\Big)^{1/2}\notag\\
&\le C o(t)(\mathcal{E}^n_t+\mathcal{E}^{n+1}_t),
\end{align} 
as $t\rightarrow 0$.
Next, we deal with the third and the fourth terms on the right-hand side of \eqref{ineq: Resultant 2}. For brevity we define
\begin{align*}
&\sum_{k=1}^3 \sup_v \langle v\rangle^{2\beta} \vert (\Delta_j \Gamma_{loss}^k (f^n, f^{n+1}), \Delta_j f^{n+1})_x \vert\ =: \sum_{k=1}^3 \mathsf{L}^k_j,\\
&\overline{\sum} \Big( \int^t_0 \mathsf{L}^k_j ds\Big)^{1/2}=: L^k,
\end{align*}
and  we also define $\mathsf{G}^k_j$ and $G^k$ in the same way. The direct computations imply that 
\begin{equation}
\label{ineq: The Integrands}
\begin{aligned}
&\mathsf{L}^1_j\le C \sum_{\vert i-j\vert \le 4} \Vert \Delta_i f^n\Vert_{L^\infty_{\beta+\gamma/2}L^2_x} \Vert f^{n+1}\Vert_{L^\infty_{\beta+\gamma/2}L^\infty_x}\Vert \Delta_j f^{n+1}\Vert_{L^\infty_{\beta+\gamma/2}L^2_x},\\
&\mathsf{L}^2_j \le C \sum_{\vert i-j\vert \le 4} \Vert f^n\Vert_{L^\infty_{\beta+\gamma/2}L^\infty_x} \Vert \Delta_i f^{n+1}\Vert_{L^\infty_{\beta+\gamma/2}L^2_x} \Vert \Delta_j f^{n+1}\Vert_{L^\infty_{\beta+\gamma/2}L^2_x},\\
&\mathsf{L}^3_j\le C \sum_{i\ge j-3} \Vert f^n\Vert_{L^\infty_{\beta+\gamma/2}L^\infty_x} \Vert \Delta_i f^{n+1}\Vert_{L^\infty_{\beta+\gamma/2}L^2_x} \Vert \Delta_j f^{n+1}\Vert_{L^\infty_{\beta+\gamma/2}L^2_x},\\
&\mathsf{G}^k_j\le C \sum_{\vert i-j\vert \le 4} \Vert f^n\Vert_{L^\infty_{\beta+\gamma/2}L^\infty_x} \Vert \Delta_i f^n\Vert_{L^\infty_{\beta+\gamma/2}L^2_x} \Vert \Delta_j f^{n+1}\Vert_{L^\infty_{\beta+\gamma/2}L^2_x},\ k=1, 2,\\
&\mathsf{G}^3_j\le C \sum_{\i\ge j-3} \Vert f^n\Vert_{L^\infty_{\beta+\gamma/2}L^\infty_x} \Vert \Delta_i f^n\Vert_{L^\infty_{\beta+\gamma/2}L^2_x} \Vert \Delta_j f^{n+1}\Vert_{L^\infty_{\beta+\gamma/2}L^2_x}.
\end{aligned}
\end{equation}
Here we give the proof of the first inequality only, and the others can be similarly obtained. In fact, for $\mathsf{L}^1_j$, it holds that 
\begin{align*}
\mathsf{L}^1_j&\le \sup_v \langle v\rangle^{2\beta}\int_{\mathbb{R}^d}\int_{\mathbb{R}^d\times\mathbb{S}^{d-1}} \sum_{\vert i-j\vert\le 4} \vert v-v_*\vert^\gamma b_0(\theta) M_*^{1/2}  \vert \Delta_j (\Delta_i f_*^n S_{i-1}f^{n+1})\vert \vert \Delta_j f^{n+1}\vert dv_*d\omega dx \\
&\le C\sup_v \langle v\rangle^{2\beta}\sum_{\vert i-j\vert\le 4}\int_{\mathbb{R}^d}\vert v-v_*\vert^\gamma M_*^{1/2} \Vert \Delta_i f^n_* \Vert_{L^2_x} \Vert S_{i-1} f^{n+1}\Vert_{L^\infty_x} \Vert \Delta_j f^{n+1}\Vert_{L^2_x} dv_*\\
&\le C \sum_{\vert i-j\vert \le 4} \Vert \Delta_i f^n\Vert_{L^\infty_{\beta+\gamma/2}L^2_x} \Vert f^{n+1}\Vert_{L^\infty_{\beta+\gamma/2}L^\infty_x}\Vert \Delta_j f^{n+1}\Vert_{L^\infty_{\beta+\gamma/2}L^2_x},
\end{align*}
where we have used the estimate 
\begin{align*}
\sup_v \langle v\rangle^{2\beta} \int_{\mathbb{R}^3} \langle v_*\rangle^{-\beta-\gamma/2} \langle v\rangle^{-2\beta-\gamma} \vert v-v_*\vert^\gamma M_*^{1/2} dv_*\le C,
\end{align*}
in terms of Lemma \ref{equiv: v_* Integral}.
Therefore, by \eqref{ineq: The Integrands}, we have
\begin{equation}
\label{ineq: Trilinear estimate of Gamma}
\begin{aligned}
&L^1 \le C (1+t)^{-\alpha} \sqrt{\mathcal{E}^{n+1}_t}\sqrt{\mathcal{D}^n_t}\sqrt{\mathcal{D}^{n+1}_t},\\
&L^k \le C (1+t)^{-\alpha} \sqrt{\mathcal{E}^n_t}\mathcal{D}^{n+1}_t,\ k=2,3,\\
&G^k \le C (1+t)^{-\alpha} \sqrt{\mathcal{E}^n_t}\sqrt{\mathcal{D}^n_t}\sqrt{\mathcal{D}^{n+1}_t}, k=1,2,3.
\end{aligned}
\end{equation}
Once again we only show the first estimate on $L^1$ in \eqref{ineq: Trilinear estimate of Gamma}. In fact, by the inclusion $B^s_{2,1}\hookrightarrow L^\infty$ for $s\ge d/2$ and the negativity of $\gamma$, it holds that 
\begin{align*}
&L^1 \le C\overline{\sum} \sup_{0\le s\le t} (1+s)^\alpha \Vert f^{n+1}(s)\Vert_{\tilde{L}^\infty_{\beta+\gamma/2}(B^s_{2,1})}\\
&\qquad\qquad\qquad\times \Big( \int^t_0 \sum_{\vert i-j\vert\le 4} \Vert \Delta_i f^n \Vert_{L^\infty_{\beta+\gamma/2}L^2_x} \Vert \Delta_j f^{n+1} \Vert_{L^\infty_{\beta+\gamma/2}L^2_x} (1+s)^{-\alpha} ds \Big)^{1/2}\\
&\le C \sqrt{\mathcal{E}^{n+1}_t} \overline{\sum} \Big( \int^t_0 \sum_{\vert i-j\vert\le 4} \Vert \Delta_i f^n \Vert_{L^\infty_{\beta+\gamma/2}L^2_x}^2 ds \Big)^{1/4}\Big( \int^t_0  \Vert \Delta_j f^{n+1} \Vert_{L^\infty_{\beta+\gamma/2}L^2_x}^2  ds \Big)^{1/4}\\
&\le C \sqrt{\mathcal{E}^{n+1}_t} \Big( \overline{\sum}\Big( \int^t_0 \sum_{\vert i-j\vert\le 4} \Vert \Delta_i f^n \Vert_{L^\infty_{\beta+\gamma/2}L^2_x}^2 ds \Big)^{1/2}\Big)^{1/2} \Big( \overline{\sum}\Big( \int^t_0  \Vert \Delta_j f^{n+1} \Vert_{L^\infty_{\beta+\gamma/2}L^2_x}^2 ds \Big)^{1/2}\Big)^{1/2}\\
&\le C (1+t)^{-\alpha}\sqrt{\mathcal{E}^{n+1}_t}\sqrt{\mathcal{D}^n_t}\sqrt{\mathcal{D}^{n+1}_t}.
\end{align*}
Now, substituting \eqref{ineq: Trilinear Estimate of K} and \eqref{ineq: Trilinear estimate of Gamma} back to \eqref{ineq: Resultant 2}, we have
\begin{align*}
&\overline{\sum} \Vert \Delta_j f^{n+1} (t)\Vert_{L^\infty_\beta L^2_x} + \overline{\sum} \Big( \int^t_0 \Vert \Delta_j f^{n+1} (t)\Vert_{L^\infty_{\beta+\gamma/2}L^2_x}^2ds\Big)^{1/2}\notag\\
&\le C\overline{\sum} \Vert \Delta_j f_0\Vert_{L^\infty_\beta L^2_x} + C o(t)(\mathcal{E}^n_t+\mathcal{E}^{n+1}_t) + C (1+t)^{-\alpha} \sqrt{\mathcal{E}^{n+1}_t}\sqrt{\mathcal{D}^n_t}\sqrt{\mathcal{D}^{n+1}_t}\\
&\quad+ C (1+t)^{-\alpha} \sqrt{\mathcal{E}^n_t}\mathcal{D}^{n+1}_t + C (1+t)^{-\alpha} \sqrt{\mathcal{E}^n_t}\sqrt{\mathcal{D}^n_t}\sqrt{\mathcal{D}^{n+1}_t}.
\end{align*}
Multiplying the above inequality by $(1+t)^\alpha$, and then taking supremum in $t$ over $[0, T]$, it follows that
\begin{align*}
\mathcal{E}^{n+1}_T + \mathcal{D}^{n+1}_T&\le C\mathcal{E}^{0}_T + C o(T)(\mathcal{E}^n_T+\mathcal{E}^{n+1}_T) + C  \sqrt{\mathcal{E}^{n+1}_T}\sqrt{\mathcal{D}^n_T}\sqrt{\mathcal{D}^{n+1}_T}\\
&\quad+ C \sqrt{\mathcal{E}^n_T}\mathcal{D}^{n+1}_T + C \sqrt{\mathcal{E}^n_T}\sqrt{\mathcal{D}^n_T}\sqrt{\mathcal{D}^{n+1}_T}.
\end{align*}
Notice that we have used the fact that $\mathcal{E}_T(\cdot)$ and $\mathcal{D}_T(\cdot)$ are non-decreasing in $T$.
We fix a small constant $\eta>0$, and further dominate the last three terms on the right-hand side by
\begin{align*}
\eta \mathcal{D}^{n+1}_T +\frac{C}{\eta}\mathcal{E}^n_T\mathcal{D}^n_T + C\sqrt{\mathcal{D}^n_T}(\mathcal{E}^{n+1}_T+\mathcal{D}^{n+1}_T).
\end{align*}
Then, there is a constant $C>0$ independent of $n$ such that 
\begin{align*}
(1-o(T)-CM_0^{1/2})\mathcal{E}^{n+1}_T +(1-\eta-2CM_0^{1/2}) \mathcal{D}^{n+1}_T \le C (M_0 +o(T)M_0+\eta^{-1}M_0^2).
\end{align*}
By further taking $\eta>0$ small, $M_0>0$ small, and $T_*>0$ small in order, we then prove \eqref{ad.ap.p1} with $n$ replaced by $n+1$. Therefore,  by induction argument, \eqref{ad.ap.l1} holds true for all $n$. This completes the proof of Lemma \ref{lem: Approximation functions}.
\end{proof}

With the aid of the approximation functions, we shall prove the local-in-time existence. We remark that in the hard potential case, the similar local-in-time existence result also holds true and thus the unique solution in the mild form indeed can be improved to be the unique strong solution in the sense of distributions.

\begin{thm}\label{ap.thm}
Under the same assumptions of Lemma \ref{lem: Approximation functions}, there are $M_0>0$ and $T^\ast>0$ such that if  initial datum $f_0$ satisfies 
\begin{align*}
\Vert f_0\Vert_{\tilde{L}^\infty_\beta (B^s_{2,1})}\le M_0,
\end{align*}
then the Cauchy problem \eqref{eqn: BE near M} on the Boltzmann equation admits a unique local-in-time mild solution $f(t,x,v)$ in $L^\infty(0,T_*; \tilde{L}^\infty_\beta (B^s_{2,1}))$ satisfying
\begin{align*}
\mathcal{F}_T(f)\le 2M_0
\end{align*}
for any $T\in [0, T_*)$, where $\mathcal{F}_T(f)$ is continuous with respect to $T\in [0, T_*)$. Moreover, the non-negativity of solutions can be preserved in the sense that if $F_0(x,v) =M +M^{1/2}f_0(x,v) \ge 0$, then so is $F(t,x,v)=M+M^{1/2}f(t,x,v)$.
\end{thm}

\begin{proof}
First we consider the uniqueness. 
Suppose that $f$ and $g$ are two solutions to the Cauchy problem \eqref{eqn: BE near M}
with the same initial data $f|_{t=0}=f_0=g|_{t=0}$. Taking difference of the equations for $f$ and $g$ gives
\begin{align*}
(\partial_t +v\cdot \nabla_x)(f-g) +\nu  (f-g) =\Gamma(f-g,f) +\Gamma(g, f-g) +K(f-g).
\end{align*}
Here, for brevity we did not directly make use of the integral form of equations, as the solution can be explained to be a strong solution in the sense of distributions, see \cite{DL}.
The same procedure carried out in the proof of Lemma \ref{lem: Approximation functions} shows
\begin{align*}
\mathcal{F}_T(f-g)&\le C \sqrt{\mathcal{E}_T(f)+\mathcal{E}_T(g)}\mathcal{D}_T(f-g) \\
&+ C\sqrt{\mathcal{E}_T(f-g)}\sqrt{\mathcal{D}_T(f)+\mathcal{D}_T(g)}\sqrt{\mathcal{D}_T(f-g)} + o(T)\mathcal{E}_T(f-g)\\
&\le C\left(\sqrt{M_0}+o(T)\right)\mathcal{F}_T(f-g).
\end{align*}
Thus, one has $f-g\equiv 0$ by taking $M_0>0$ and $T_*>0$ to be further suitably small, if necessary. This proves the uniqueness. 

Next, we  show the continuity of $\mathcal{F}_T(f)$ in $T$. Note that the continuity of $\overline{\sum} \Vert \Delta_j f(t)\Vert_{L^\infty_\beta L^2_x}$ is a consequence of the following fact that 
\begin{align}\label{ad.ap.p2}
\lim_{t_2\rightarrow t_1} \overline{\sum} \Big( \int^{t_2}_{t_1} \Vert \Delta_j f \Vert_{L^\infty_{\beta+\gamma/2}L^2_x}^2 dt\Big)^{1/2}=0,
\end{align}
for any $0\le t_1, t_2 <T_*$. Indeed, we may assume $t_1<t_2$ without loss of generality. Starting from \eqref{eqn: Resultant 1} again, one can show that
\begin{align*}
\vert \mathcal{E}_{t_2}(f)-\mathcal{E}_{t_1}(f)\vert \le C(\sqrt{M_0}+1)\overline{\sum} \Big( \int^{t_2}_{t_1} \Vert \Delta_j f \Vert_{L^\infty_{\beta+\gamma/2}L^2_x}^2 dt\Big)^{1/2}.
\end{align*}
Thus, it remains to show \eqref{ad.ap.p2}. 
Take $\varepsilon>0$. In terms of the finiteness of $\mathcal{D}_T(f)$, there is an integer $N$ large enough such that
\begin{align*}
\sum_{j\ge N+1}2^{js} \Big( \int^{t_2}_{t_1} \Vert \Delta_j f \Vert_{L^\infty_{\beta+\gamma/2}L^2_x}^2 dt\Big)^{1/2} \le \sum_{j\ge N+1}2^{js} \Big( \int^T_0 \Vert \Delta_j f \Vert_{L^\infty_{\beta+\gamma/2}L^2_x}^2 dt\Big)^{1/2}< \frac{\varepsilon}{2}.
\end{align*}
Also, since $\sum_{-1\le j\le N}$ is a finite sum, there is $\delta>0$ such that if $\vert t_2-t_1\vert <\delta$, then it holds that
\begin{align*}
\sum_{-1\le j\le N}2^{js} \Big( \int^{t_2}_{t_1} \Vert \Delta_j f \Vert_{L^\infty_{\beta+\gamma/2}L^2_x}^2 dt\Big)^{1/2} <\frac{\varepsilon}{2}.
\end{align*}
Therefore, whenever $\vert t_2-t_1\vert <\delta$, it holds that 
$$
\overline{\sum} \Big( \int^{t_2}_{t_1} \Vert \Delta_j f \Vert_{L^\infty_{\beta+\gamma/2}L^2_x}^2 dt\Big)^{1/2}\leq \left(\sum_{j\ge N+1} +\sum_{-1\le j\le N}\right) \{\cdots\}<\frac{\varepsilon}{2}+\frac{\varepsilon}{2}=\varepsilon.
$$ 
This then proves \eqref{ad.ap.p2}. Furthermore, for $t_\alpha :=\max \{(1+t_1)^\alpha, (1+t_2)^\alpha\}$, one has 
\begin{align*}
0\le \mathcal{F}_{t_2}(f)-\mathcal{F}_{t_1}(f)& = \Big(\mathcal{E}_{t_2}(f)-\mathcal{E}_{t_1}(f)\Big)+\Big(\mathcal{D}_{t_2}(f)-\mathcal{D}_{t_1}(f)\Big),\\
\mathcal{E}_{t_2}(f)-\mathcal{E}_{t_1}(f) &\le \sup_{t_1\le t\le t_2} (1+t)^{\alpha} \overline{\sum}\Vert \Delta_j f(t)\Vert_{L^\infty_\beta L^2_x} \le t_\alpha \sup_{t_1\le t\le t_2}\overline{\sum}\Vert \Delta_j f(t)\Vert_{L^\infty_\beta L^2_x}\rightarrow 0,\\
\mathcal{D}_{t_2}(f)-\mathcal{D}_{t_1}(f)&\le t_\alpha \overline{\sum} \Big( \int^{t_2}_0 \Vert \Delta_j f \Vert_{L^\infty_{\beta+\gamma/2}L^2_x}^2 dt\Big)^{1/2} -t_\alpha\overline{\sum}\Big( \int^{t_1}_0 \Vert \Delta_j f \Vert_{L^\infty_{\beta+\gamma/2}L^2_x}^2 dt\Big)^{1/2}\\
&\le t_\alpha \overline{\sum} \Big( \int^{t_2}_{t_1} \Vert \Delta_j f \Vert_{L^\infty_{\beta+\gamma/2}L^2_x}^2 dt\Big)^{1/2} \rightarrow 0
\end{align*}
as $t_2\rightarrow t_1$. Thus, the continuity of $\mathcal{F}_T(f)$ in $T$ is proved.

For the non-negativity of solutions, see  \cite[pp.416--417]{DHWY}, for instance, and details are omitted for brevity. This completes the proof of Theorem \ref{ap.thm}. 
\end{proof}

\subsection*{Acknowledgement}
The first author is partially supported by the General Research Fund (Project No.\ 14302716) from RGC  of Hong Kong. The second author is partially supported by JSPS Kakenhi Grant (No.\ 16J03963). Also, he shows his deep gratitude to Professor Renjun Duan for his kind hospitality when he visited the Chinese University of Hong Kong from July to August in 2017.

\end{document}